\documentclass[12pt,reqno]{amsart}
\usepackage{amssymb,amsmath,amsthm,amsfonts,fancyhdr}
\usepackage{tikz}
\usepackage{indentfirst}
\usepackage{mathrsfs}
\usepackage{setspace}
\usepackage{color}
\usepackage{cite}
\usepackage{bm}
\usepackage{geometry}
\numberwithin{equation}{section}

\newtheorem{thm}{Theorem}[section]
\newtheorem{defn}[thm]{Defnition}
\newtheorem{lemma}[thm]{Lemma}
\newtheorem{cor}[thm]{Corollary}
\newtheorem{pro}[thm]{Proposition}
\newtheorem{re}[thm]{Remark}

\renewcommand{\det}{\mbox{det}}
  
  \newcommand{\dist}{\mbox{dist}}

  \numberwithin{equation}{section}
  \numberwithin{figure}{section}

\begin{document}
\title[Global Solutions]{Global solutions to the fully nonlinear obstacle problem }
\author{Yi Cao}
\author{Dongsheng Li}
\begin{abstract} In this paper, we investigate global solutions to the fully nonlinear obstacle problem. We prove that, after factoring out invariant directions, the coincidence
set of any global solution belongs to one of the following four types: a single point, a compact convex set with nonempty interior, a ray for $n=1$, or a slender body for $n\geq 2$, where $n$ denotes the spatial dimension.

Furthermore, we establish a bijection between the collection of all global solutions with compact coincidence sets and a family of functions describing the asymptotic behavior of such solutions at infinity.
\end{abstract}
\footnotetext{This research is supported by NSFC 12071365.\newline}

\footnotetext{Yi Cao\newline
              Caoyi@snnu.edu.cn\newline
              School of Mathematics and Statistics, Shaanxi Normal University, Xi'an, P.R. China 710062.\newline}

\footnotetext{{Dongsheng Li\newline
              lidsh@mail.xjtu.edu.cn\newline
              School of Mathematics and Statistics, Xi'an Jiaotong University, Xi'an, P.R. China 710049.}}
\maketitle
\section{Introduction}

In this paper, we consider the fully nonlinear obstacle problem as follows:
\begin{equation}\label{E1.1}
\left\{\begin{split}
F(D^2 u)\leq 1\quad &\mbox{in}\quad\mathbb{R}^n,\\
u\geq 0\quad&\mbox{in}\quad\mathbb{R}^n,\\
F(D^2 u)= 1\quad &\mbox{in}\quad \{u>0\},
\end{split}\right.
\end{equation}
where $u\in C(\mathbb{R}^n)$ is understood in the viscosity sense. Throughout this paper, we impose the following assumptions on the operator
$F: \mathcal{S}_{n \times n}\rightarrow \mathbb{R}$:
\newline
(H1) $F$ is uniformly elliptic;
\newline
(H2) $F(0)=0$;
\newline
(H3) $F$ is $C^{1,\alpha_{F}}$ for some $\alpha_{F}\in(0,1]$.

Given a global solution $u$, its zero set is called the coincidence set, denoted by
\begin{equation}\label{E1.2}
\Omega^c(u):=\{x:u(x)=0\}.
\end{equation}
The set $\Omega^c(u)$ is an unknown set, and $\Gamma(u):=\partial\{u>0\}$ is called the free boundary.

Over the past few decades, the obstacle problem with a nonlinear operator has attracted great interest in the PDE community\cite{FS01, IM, Lee01, Lee02, SH}.
Most of these papers investigate local properties of solutions. In this paper, we focus on global solutions defined on the whole space.

Notice that if the operator $F=\triangle$, then the obstacle problem (\ref{E1.1}) is known as
\emph{the classical obstacle problem} as follows:
\begin{equation}\label{E1.3}
\triangle u=\chi(\{u>0\}),\quad u\geq 0, \quad \mbox{in} \quad\mathbb{R}^n.
\end{equation}
We refer to \cite{Ca01, PSU} for an introduction to results concerning the classical obstacle problem.

The research on the global solutions to the classical obstacle problem (\ref{E1.3}) has a history of nearly a century. In works \cite{Sh, KM01}, it was conjectured that
\emph{the coincidence set of a global solution to the classical obstacle problem (\ref{E1.3}) is either a half-space, an ellipsoid, a paraboloid, or a cylinder with
an ellipsoid or a paraboloid as base if it has non-empty interior. }

Under the assumption that the coincidence sets of global solutions are compact, the previous conjecture was confirmed in three dimensional space \cite{Di, Le01}. For higher dimensions, still under the compactness assumption, the conjecture was verified in \cite{DF, FS02}. A brief proof of the known result in any dimension was provided in \cite{EW}.

Removing the compactness assumption, the conjecture was proven in \cite{Sa} when the spatial dimension $n=2$. Assuming that the coincidence sets are contained between two
parallel hyperplanes, all possible solutions are provided in \cite{KM01}. Recently, the above conjecture was proven for $n\geq 6$ \cite{ESW}, and for all dimensions \cite{EFW}.
For $n=2$, avoiding the complex analysis approach, a new proof in the spirit of reference \cite{EFW} was founded in \cite{SW}.

Since all possible global solutions to the classical obstacle problem are known, some scholars naturally began to consider global solutions to the fully nonlinear obstacle
problem (\ref{E1.1}). For $n\geq 3$, under the assumption that the operator $F$ is convex, Eberle-Yu \cite{EY} characterized all possible global solutions with compact coincidence sets. To the best of the authors knowledge, a classification of global solutions to the fully nonlinear obstacle problem (\ref{E1.1}) remains open.

Our main results are Theorem \ref{thm:1}, Theorem \ref{thm:5}, and Corollary \ref{cor:3}. To simplify our statement, we introduce the following definition.

\begin{defn}
 \label{def:global solution}
 (i) $\mathcal{G}_c:=\big\{$ $u$: $u$ satisfies the obstacle problem (\ref{E1.1}), and $\Omega^c(u)$ is compact.$\big\}$

 (ii) $\mathcal{G}_{\infty}:=\big\{$ $u$: $u$ satisfies the obstacle problem (\ref{E1.1}) in $\mathbb{R}^n$, where $n\geq 2$. For such $u$,
 there exists a unique unit vector $e\in\partial B_1$, such that $u(\cdot+te)\leq u(\cdot)$ for $t\in\mathbb{R}_+$.$\big\}$

(iii) $\mathcal{Q}_c:=\big\{q(x)$ : $q$ is a function satisfies (\ref{E1.23})-(\ref{E1.24}). $\big\}$

\begin{equation}\label{E1.23}
q(x)=\left\{\begin{split}
&\frac{1}{2}(x-x_0)^TA(x-x_0)+D,\qquad n\geq 3,\\
&\frac{1}{2}(x-x_0)^TA(x-x_0)+D\log|P(x-x_0)|,\qquad n=2,
\end{split}\right.
\end{equation}
where $A\in\mathcal{S}_{n\times n}$ is a positive definite matrix, $x_0\in \mathbb{R}^n$, and $D\leq 0$. Furthermore, we have $F(A)=1$.

The positive definite matrix $P$ is determined as follows:
\begin{equation}\label{E1.24}
P=\left(\sqrt{[F_{ij}(A)]_{2\times 2}}\right)^{-1},
\end{equation}
where $F_{ij}(\cdot)$ denotes the derivatives of the operator $F$.
\end{defn}

\begin{thm}
 \label{thm:1}
 Let $F$ be a convex operator satisfying assumptions (H1)-(H3), and let $u$ be a global solution to the fully nonlinear obstacle problem (\ref{E1.1}). Suppose that the free boundary $\Gamma(u)\neq\emptyset$.

 Then $u$ is convex. Furthermore, one of the following cases holds:

 (i) $\Omega^{c}(u)$ is a single point in $\mathbb{R}^n$, and $u$ is a quadratic polynomial of the form
\begin{equation*}
u(x)=\frac{1}{2}(x-x_0)^TA(x-x_0),
\end{equation*}
where $A\in\mathcal{S}_{n\times n}$ is a positive definite matrix.

(ii) $u\in\mathcal{G}_c$ with $\mathrm{Int}\,\Omega^{c}(u)\neq\emptyset$. In addition, $\Omega^{c}(u)$ is a $C^{1,\alpha}$ domain for some exponent $\alpha\in(0,1)$.

(iii) $n=1$, and $\Omega^{c}(u)$ is a ray in $\mathbb{R}^1$. After a suitable translation and rotation of the coordinate system, $u$ can be rewritten as
\begin{equation}\label{E1.5}
u(x)=\frac{\gamma}{2}x_{+}^{2},
\end{equation}
where the constant \(\gamma\in [1/\Lambda,1/\lambda]\) satisfies \(F(\gamma)=1\).

(iv) $u\in\mathcal{G}_{\infty}$. After a suitable translation and rotation of the coordinate system, $u$ satisfies the following properties (a)-(c).

  $\quad$(a) $\partial_n u>0$ in $\Omega(u)$, and \(0\in\Omega^{c}(u)\subseteq\{x\in\mathbb{R}^{n}: x_n\leq 0\}.\)

\(\quad\)(b) There exists a unique positive definite matrix \(A_{n-1}\in\mathcal{S}_{(n-1)\times (n-1)}\) such that
 \begin{equation}\label{E1.13}
 u_{r}(x',x_n)\rightarrow\frac{1}{2}x'^TA_{n-1}x'\quad\text{in}\quad C^{1,\alpha}_{loc}(\mathbb{R}^{n}),\quad\text{as}\quad r\rightarrow +\infty,
 \end{equation}
for all $\alpha\in(0,1)$ (see Notations 9). Furthermore,
 \begin{equation}\label{E1.25}
F\left(
\begin{gathered}
\begin{matrix}
A_{n-1} & 0 \\ 0 & 0
\end{matrix}
\end{gathered}\right)=1.
 \end{equation}

  $\quad$(c) For any \(\rho>0\), define the sliced non-contact set
  \begin{equation}\label{E1.19}
  \Omega^{c}_{\rho}(u):=\Omega^{c}(u)\cap\{x_n=-\rho\}.
  \end{equation}
The following relations hold:
\begin{equation}\label{E1.18}
\left\{\begin{split}
&\lim\limits_{\rho\rightarrow+\infty}\mbox{diam}(\Omega^{c}_{\rho}(u))=+\infty,\\
&\lim\limits_{\rho\rightarrow+\infty}\mbox{diam}(\Omega^{c}_{\rho}(u))/\rho=0,\\
&\lim\limits_{\rho\rightarrow+\infty}\mbox{diam}\big(\Omega^{c}_{\rho}(u)\big)/\mathrm{MD}\big(\Omega^{c}_{\rho}(u)\big)= C,
\end{split}\right.
\end{equation}
where the constant $C$ depends only on the operator $F$ and the matrix \(A_{n-1}\).

(v) $u$ is a lower dimensional global solution, that is, there exists a unit vector $e\in\partial B_1$, such that $u(\cdot+te)\equiv u(\cdot)$ for all $t\in\mathbb{R}$.
\end{thm}

\begin{thm}
 \label{thm:5}
  Assume that the spatial dimension $n\geq 2$. Let $F$ be a convex operator satisfying assumptions (H1)-(H3). Then there exists a bijection $\Phi:\mathcal{G}_c\rightarrow \mathcal{Q}_c$
 such that
\begin{equation}\label{E1.12}
u(x)=\left\{\begin{split}
\Phi(u)&+O(|x|^{2-n}),\qquad n\geq 3,\\
\Phi(u)&+C+O(|x|^{-1}),\quad n=2,
\end{split}\right.
\quad\mbox{as}\quad |x|\rightarrow +\infty,
\end{equation}
where $C$ is a unique constant depending only on $\Phi(u)$.
\end{thm}

\begin{re}
\label{re:1}
   For the spatial dimension $n\geq 3$, Eberle-Yu \cite{EY} established the bijection $\Phi:\mathcal{G}_c\rightarrow \mathcal{Q}_c$. For $n=1$, if $u\in\mathcal{G}_c$, then we may assume  $\Omega^{c}(u)=[a, b]$ since the global solution $u$ is convex. Then $u$ is of the form
\begin{equation}\label{E1.6}
u(x)=\frac{\gamma}{2}[(x-a)_{-}^{2}+(x-b)_{+}^{2}],
\end{equation}
where $F(\gamma)=1$ and $\gamma\in [1/\Lambda,1/\lambda]$.
\end{re}

In Theorem \ref{thm:1}, if we replace the convexity of $F$ by concavity, we obtain the same results provided we strengthen the assumptions to require
that the global solution is convex. For $n=2$, invoking Proposition \ref{pro:2.12}, we can remove the assumption that $u$ is convex.

\begin{thm}
\label{cor:3}
Let $F$ be a concave operator satisfying assumptions (H1)-(H3), and let $u$ be a global solution to the fully nonlinear obstacle problem (\ref{E1.1}).
Assume that the free boundary $\Gamma(u)\neq\emptyset$. Assume additionally that $u$ is convex. Then all conclusions in Theorem \ref{thm:1} hold.
\end{thm}

\textbf{Notations}

We introduce the following notations which will be frequently used in this paper.

(1) $x=(x',x_n)=(x_1, x_2,\cdots,x_{n-1}, x_n)\in \mathbb{R}^n$.

(2) $B_{r}(x_0)=\{x\in\mathbb{R}^{n}:|x-x_0|<r\}$, and $B_{r}=B_{r}(0)$.

  $\qquad B'_{r}(x'_0)=\{x'\in\mathbb{R}^{n-1}:|x'-x'_0|<r\}$, and $B'_{r}=B'_{r}(0)$.

$\qquad\mathcal{C}_{r,l}(x'_0,x^n_0)=B'_{r}(x'_0)\times (x^n_0-l,x^n_0+l).$

(3) $\mathcal{S}_{n\times n}$ is the space of real $n\times n$ symmetric matrices. $I_{n}$ is the identity matrix. Given $N\in\mathcal{S}_{n\times n}$,
we say $N\geq(>) 0$ whenever $N$ is a nonnegative (positive) definite symmetric matrix.

(4) $F$ is uniformly elliptic if there exist constants $0<\lambda\leq\Lambda<+\infty$ such that
\begin{equation*}
\lambda||N||\leq F(A + N)-F(A)\leq \Lambda ||N||
\end{equation*}
for all $A, N\in\mathcal{S}_{n\times n}$ with $N\geq 0$.

(5) Let $f$ be a continuous function in $\Omega$. We say that $u\in S(\lambda, \Lambda, f)$ if $u$ is a continuous function in $\Omega$, such that
\begin{equation*}
\mathcal{M}^+(D^2u(x),\lambda, \Lambda)\geq f(x),\qquad\mathcal{M}^-(D^2u(x),\lambda, \Lambda)\leq f(x),
\end{equation*}
in the viscosity sense in $\Omega$, where $\mathcal{M}^+,\mathcal{M}^-$ denote the Pucci extremal operators.

(6) $\{e_i\}_{i=1}^{n}$ are the standard basis of $\mathbb{R}^n$, i.e., $e_i=(0,\cdots,0, \overset{i\text{-th}}{1},0,\cdots,0).$

(7) Let $e,\tau\in\mathbb{R}^n$ be two vectors. We denote by $\angle\vartheta(e,\tau)$ the angle formed by them.

(8) Let $P$ be a point in $\mathbb{R}^n$. We define a cone $V_P(e,\theta)$ as follows:
\begin{equation*}
V_P(e,\theta):=\{x\in\mathbb{R}^n: 0\leq\angle\vartheta(e,x-P)<\theta\},\quad\forall e\in\partial B_1,\quad\forall\theta\in(0,\pi].
\end{equation*}

(9) Let $x_0\in\mathbb{R}^n$. Given $r>0$, we define the rescaled function $u_{r}$ by
\begin{equation*}
u_{r}(x,x_0):=r^{-2}u(rx+x_0).
\end{equation*}
Moreover, we omit $x_0$ provided $x_0=0$.

(10) Given a measurable set $E\subseteq\mathbb{R}^{n}$, $|E|$ denotes its Lebesgue measure.

(11) Let $E\subseteq\mathbb{R}^{n}$ be a measurable set. We denote by $\mathrm{MD}(E)$ the minimal distance between two parallel hyperplanes containing $E.$

(12) Fixing $r>0$, we define the thickness of the set $\{u=0\}$ in $B_r(x_0)$ as follows:
\begin{equation}\label{E2.5}
\delta_r(u,x_0):=\frac{\mathrm{MD}(\{u=0\}\cap B_r(x_0))}{r}.
\end{equation}
We note the scaling invariance property $\delta_1(u_{r}(x,x_0),0)=\delta_r(u,x_0)$.

The paper is organized as follows. In Section 2, we collect some useful preliminaries for fully nonlinear obstacle problems. Section 3 is devoted to the proof of Theorem 1.3. In Section 4, we prove that the coincidence set $\Omega^{c}(u)$ is slender whenever $u\in\mathcal{G}_{\infty}$. We complete the proof of Theorem 1.2 in Section 5. Section 6 contains a sketch of the proof of Theorem 1.5.

\section{Preliminaries}

In this section, we introduce several properties of the fully nonlinear obstacle problem. In Proposition \ref{pro:2.11}, the inclusions (1)-(2) are taken from \cite{Lee01,Lee02}.
 We refer to \cite{FS01} or \cite{Lee02} for the interior $C^1$ regularity of the free boundary, and to \cite{RT} for the interior $C^{1,\alpha}$ regularity.
 We say that $u$ is a viscosity solution to the fully nonlinear obstacle problem in $B_1$ if we replace $\mathbb{R}^{n}$ with $B_1$ in (\ref{E1.1}).

\begin{pro}
 \label{pro:2.11}
 Let the operator $F$ be convex, and let $u$ be a viscosity solution to the fully nonlinear obstacle problem in $B_1$. Then the following results hold.

 (1) Nondegeneracy. If $x\in\overline{\Omega(u)}$ and $B_{r}(x)\subset\subset B_1$, then
\begin{equation}\label{E2.15}
\sup\limits_{y\in \partial B_{r}(x)}(u(y)-u(x))\geq C(n,\Lambda)r^2;
\end{equation}

(2) Interior $C^{1,1}$ estimates. If $u(0)=0$, then
\begin{equation}\label{E2.12}
\|D^{2}u\|_{L^{\infty}(B_{1/2})}\leq C(n,\lambda, \Lambda);
\end{equation}

(3) Local  $C^1$ regularity and local $C^{1,\alpha}$ regularity of $\Gamma(u)$. Suppose there exists a constant $\epsilon_0>0$ such that
\begin{equation}\label{E2.20}
\delta_r(u,y)\geq\epsilon_0\qquad \forall y\in B_{1/2}(0)\cap \Gamma(u),\qquad \forall r<1/4.
\end{equation}
Then $B_{1/4}(0)\cap\Gamma(u)$ is a $C^1$ manifold.
Furthermore, if the free boundary $\Gamma(u)\cap B_{1/2}$ is a $C^{1}$ manifold,
then $\Gamma(u)\cap B_{1/4}$ is a $C^{1,\alpha}$ manifold, where $\alpha\in(0,1)$ is a universal constant.
\end{pro}

The next proposition is due to Lee-Shahgholian \cite{Lee02}. It shows that if a global solution has a "thick coincidence set", then it must be a half-space solution.

\begin{pro}
 \label{pro:2.12}
Let the operator $F$ be convex (or concave when the spatial dimension $n=2$), and let $u$ be a global solution to the fully nonlinear obstacle problem (\ref{E1.1}).
Then the following results hold.

(1) $u$ is convex.

 (2) Suppose there exists a constant $\epsilon_0>0$ such that
\begin{equation}\label{E1.15}
\delta_r(u, 0)>\epsilon_0>0\qquad \forall r>0,
\end{equation}
then $u$ is a half-space solution; that is, $u$ takes the form
\begin{equation}\label{E1.20}
u(x)=\frac{\gamma}{2}\big[(x, e)_{+}\big]^2,
\end{equation}
for some unit vector $e$, where $\gamma\in [1/\Lambda, 1/\lambda]$ satisfies $F(\gamma e\otimes e)=1.$
\end{pro}

The next remark is a consequence of Proposition \ref{pro:2.11} and Proposition \ref{pro:2.12}.
\begin{re}
\label{re:2.13}
Suppose that all hypotheses of Proposition \ref{pro:2.12} hold. Suppose further that the coincidence set has nonempty interior. Then all free boundary points are regular
(cf.\ \cite{PSU} for the definition of regular points.) Applying Proposition \ref{pro:2.11}, for any $x_0\in \Gamma(u)$, we see that  the free boundary $\Gamma(u)$ is a
$C^{1,\alpha}$ graph near $x_0$.
 \end{re}

 Next we introduce the following comparison principles (cf.\ \cite{EY}).

\begin{pro}
 \label{pro:2.16}
Suppose that $u$ is a global solution to the fully nonlinear obstacle problem (\ref{E1.1}).

 (1) If $v\in C(B_1)$ is a viscosity solution to the fully nonlinear obstacle problem in $B_1$ satisfying
\begin{equation*}
 v\leq u \quad\mbox{on}\quad \partial B_{1},
 \end{equation*}
 then $v\leq u$ in $B_{1}$.

(2) Suppose that $v\in C(\mathbb{R}^{n})$ solves the fully nonlinear problem (\ref{E1.1}). If
\begin{equation*}
\liminf\limits_{|x|\rightarrow +\infty}(u-v)\geq 0,
 \end{equation*}
then $v\leq u$ in $\mathbb{R}^{n}.$
\end{pro}

\begin{lemma}
 \label{lemma 3.2}
Let $F$ be a convex operator, and let $u, v\in\mathcal{G}_c$. Assume that $u>v>0$ on $\partial B_1$. Assume in addition that $\operatorname{Int}\Omega^{c}(u)\neq\emptyset$. Then
\begin{equation}\label{E3.64}
\Omega^{c}(u)\subseteq\operatorname{Int}\Omega^{c}(v).
\end{equation}
\end{lemma}

\begin{proof}
Let us assume by contradiction that (\ref{E3.64}) fails. Using Proposition \ref{pro:2.16}, we deduce that $v\leq u$ in $B_1$. This implies that
$\Omega^{c}(u)\subseteq \Omega^{c}(v)\subset\subset B_1$. Then we may choose $y\in\mathbb{R}^{n}$, such that $y\in\Gamma(u)\cap\Gamma(v)$.
After a translation of coordinates, we set $0\in\Gamma(u)\cap\Gamma(v)$. Then we have
\begin{equation}\label{E3.44}
|\nabla u|(0)=|\nabla v|(0)=0.
\end{equation}
Let $w:=u-v$. Since $F(D^2u)=F(D^2v)$ in $B_1\cap\Omega(v),$ we have
\begin{equation*}
\left\{\begin{split}
\mathcal{M}^-(\lambda/n, \Lambda, D^2w)\leq 0\qquad &\mbox{in}\qquad B_1\cap\Omega(v),\\
w>0\qquad &\mbox{on}\qquad \partial B_1,\\
w\geq 0\qquad &\mbox{on}\qquad \Gamma(v).
\end{split}\right.
\end{equation*}
The strong maximum principle implies
\begin{equation*}
v<u\qquad\mbox{in}\quad B_1\cap\Omega(v).
\end{equation*}

Applying Proposition \ref{pro:2.12}, we find that $u$ and $v$ are convex functions. By the convexity of $\Omega^{c}(v)$, there exists $z\in\Omega(v)$ and $\rho>0$, such that
\begin{equation}\label{E3.47}
B_\rho(z)\subset\Omega(v),\qquad\mbox{and}\quad \partial B_\rho(z)\cap\Omega^{c}(v)=\{0\}.
\end{equation}
Then we have
\begin{equation*}
\left\{\begin{split}
w \in S(\lambda/n, \Lambda, 0)\quad &\mbox{in}\qquad B_\rho(z),\\
w >0\qquad &\mbox{in}\qquad B_\rho(z),\\
w(0) =0.&
\end{split}\right.
\end{equation*}
The Hopf lemma gives
\begin{equation*}
\frac{\partial}{\partial \nu}w(0)>0,
\end{equation*}
where $\nu$ denotes the unit outward normal vector to $\Omega^c(v)$ at the origin (recalling the $C^1$ regularity of the free boundary). However, this contradicts (\ref{E3.44}).
\end{proof}

Let
 \begin{equation}\label{E2.45}
 \mathcal{L}u:=\operatorname{tr}\big(A(x)D^2u\big),
 \end{equation}
 where $\lambda I\leq A(x)\leq \Lambda I$. The next lemma is a direct consequence of the Hopf Lemma.

\begin{lemma}
 \label{le:2.2}
Let $u\in C(\mathcal{C}_{2}^+)$ satisfy
$$
\left\{\begin{split}
\mathcal{L}u&=0 \quad\mbox{in}\quad \mathcal{C}_{2}^+,\\
u&=0\quad\mbox{on}\quad \{x_n=0\}\cap\partial \mathcal{C}_{2}^+,\\
u&=1\quad\mbox{on}\quad \{x_n=1\}\cap\partial \mathcal{C}_{2}^+,\\
u&>0\quad\mbox{in}\quad \mathcal{C}_{2}^+,
\end{split}\right.
$$
where $\mathcal{C}_{2}^+:=\mathcal{C}_{2,2}(0)\cap\{x_n>0\}$ denotes a cylinder; see Notations (2).

Then
\begin{equation}\label{E3.22}
\partial_n u(0)\geq C,
\end{equation}
where $C>0$ is a constant depending only on $\lambda$, $\Lambda$, and the dimension $n$.
\end{lemma}

\section{Proof of Theorem \ref{thm:5}}

Assuming $u\in\mathcal{G}_c$ in $\mathbb{R}^{2}$, the following lemma constructs an asymptotic profile.

\begin{lemma}
 \label{lemma 3.1}
 Suppose all hypotheses of Theorem \ref{thm:1} hold. Suppose further that the spatial dimension $n=2$. If $u\in\mathcal{G}_c$, then there exists a unique function $q\in\mathcal{Q}_c$ such that
\begin{equation}\label{E3.41}
u(x)=q(x)+C+O(|x|^{-1})\qquad\text{as}\quad |x|\rightarrow +\infty,
\end{equation}
where $C$ is a constant.
Moreover, if $q$ is a quadratic polynomial, then $C\le 0$.
\end{lemma}

\begin{proof}
We divide the proof into two steps.

\indent\textit{Step 1. We define $q$.}

Since $\Omega^{c}(u)$ is compact, we may assume $0\in\Omega^{c}(u)\subseteq B_1$. Then
\begin{equation*}
F(D^2u)-1=0\qquad\mbox{in}\quad\mathbb{R}^{2}\backslash B_1.
\end{equation*}
In view of \cite[Theorem 1.1]{LL}, we deduce that
\begin{equation}\label{E4.6}
u(x)=\frac{1}{2}x^TAx+Bx+C+D\log|P x|+O(|x|^{-1}), \qquad\text{as}\quad |x|\rightarrow +\infty,
\end{equation}
where $A,P\in\mathcal{S}_{2\times 2}$. Recall that $P$ is defined in (\ref{E1.24}).
By the interior $C^{1,1}$ estimate, (\ref{E4.6}) implies that
\begin{equation*}
D^2u(x)\rightarrow A,\qquad\text{as}\quad |x|\rightarrow +\infty.
\end{equation*}
It follows that $F(A)=1$.

Since $u\geq 0$, we see that $A\geq 0$. After a suitable rotation of coordinates, we have
\begin{equation*}
u(x)=(a_1x_1^2+a_2x_2^2)+(b_1x_1+b_2x_2)+C+D\log|P x|+O(|x|^{-1}),\qquad\text{as}\quad|x|\rightarrow +\infty.
\end{equation*}
It is readily seen that $a_i\geq 0$ for $i=1,2$. In fact, combining $u\geq 0$ and (\ref{E4.6}), we see that if $a_i=0$, then $b_i=0$. In summary, we obtain
\begin{equation}\label{E3.42}
u(x)=\frac{1}{2}(x-x_0)^TA(x-x_0)+D\log|P (x-x_0)|+C+O(|x|^{-1}),\qquad\text{as}\quad|x|\rightarrow +\infty.
\end{equation}
We define
\begin{equation}\label{E4.55}
q(x):=\frac{1}{2}(x-x_0)^TA(x-x_0)+D\log|P (x-x_0)|.
\end{equation}
This yields (\ref{E3.41}). Uniqueness of $q$ is clear.

\indent\textit{Step 2. $q(x)\in\mathcal{Q}_c$.}

We prove $D\leq 0$ in (\ref{E4.55}). Suppose for contradiction that $D>0$. For $\epsilon>0$, consider the auxiliary function $v(x)$ defined by
\begin{equation*}
v(x):=\frac{1}{2}(x-x_0)^TA(x-x_0)+\epsilon.
\end{equation*}
For sufficiently large $R\gg1$ and small $\epsilon>0$, the condition $D>0$ implies $u\geq v$ on $\partial B_R(x_0)$. Using
\begin{equation*}
F(D^2u)\leq 1=F(D^2v),\qquad \mbox{in}\qquad B_R(x_0),
\end{equation*}
Proposition \ref{pro:2.16} then gives $u\geq v>0$ in $B_R(x_0)$. Hence $\Omega^c(u)=\emptyset$, which is a contradiction.

Before establishing positivity of $A$, we first show that if $D=0$ in (\ref{E3.42}), then $C\leq 0$. Suppose for contradiction that $D=0$ and $C>0$.
Choose $R\gg1$ and small $\epsilon>0$. Then $u\geq v>0$ on $\partial B_R(x_0)$. Using Proposition \ref{pro:2.16} again, then $u\geq v>0$ in $B_R(x_0)$, yielding a contradiction.

We now verify $A>0$. If $D^2+C^2>0$, then nonnegativity $u\ge 0$ implies
\begin{equation*}
\frac{1}{2}(x-x_0)^TA(x-x_0)>0,\qquad\text{as}\quad|x|\rightarrow +\infty.
\end{equation*}
In this case, all eigenvalues of $A$ are positive.
If $C=D=0$, then for $\epsilon>0$, we have
\begin{equation*}
u>\frac{1}{2}(x-x_0)^TA(x-x_0)-\epsilon,\qquad\text{as}\quad|x|\rightarrow +\infty.
\end{equation*}
Applying the comparison principle and letting $\epsilon\rightarrow 0^+$, we obtain $u\geq\frac{1}{2}(x-x_0)^TA(x-x_0)$.
Since $\Omega^{c}(u)\neq\emptyset$, we deduce that $u(x)=\frac{1}{2}(x-x_0)^TA(x-x_0).$ Since $\Omega^{c}(u)$ is bounded, it follows that $A>0$ in this case.
\end{proof}

\begin{lemma}
 \label{lemma 3.3}
 Suppose that all hypotheses of Lemma \ref{lemma 3.1} hold. Suppose further that $q\in\mathcal{Q}_c$ is given by (\ref{E3.41}). Then

(i) $x_0\in \mathrm{Int}\,\Omega^c(u)$ provided $D<0$ (see (\ref{E1.23}));

(ii) $u(x)=q(x)=\frac{1}{2}(x-x_0)^TA(x-x_0)$ provided $D=0$.
\end{lemma}

\begin{proof}
We show that (i) holds. Without loss of generality, we assume that $x_0=0$. Suppose on the contrary that $D<0$, and the origin $0\notin\mathrm{Int}\,\Omega^{c}(u)$. We prove
 \begin{equation}\label{E3.62}
 \mathrm{Int}\,\Omega^{c}(u)\neq\emptyset.
 \end{equation}

 In fact, if (\ref{E3.62}) is fails, then we have
 \begin{equation*}
 F(D^2 u)=1,\qquad\mbox{a.e. in}\quad\mathbb{R}^n.
\end{equation*}
 By the interior $C^{1,1}$ estimate, we observe that $u$ is a strong $L^n$ solution in $\mathbb{R}^n$. Under assumptions about the operator $F$, strong $L^n$ solutions
 are also viscosity solutions (cf.\cite{Ca02.5}). Then the Evans-Krylov theorem implies that $u$ is a quadratic polynomial. Together with (\ref{E3.41}) and (\ref{E4.55}),
 we deduce that
 \begin{equation*}
 u(x)=q(x)=\frac{1}{2}x^TAx.
 \end{equation*}
 This contradicts $D<0.$ The claim follows.

By $D<0$, we have $u<\frac{1}{2}x^TAx$ as $|x|\rightarrow +\infty$. Using Proposition \ref{pro:2.16} again, we have
\begin{equation*}
0\leq u\leq\frac{1}{2}x^TAx,\qquad\mbox{in}\quad\mathbb{R}^n.
\end{equation*}
Hence $u(0)=0$. Since $0\notin\mathrm{Int}\,\Omega^{c}(u)$, we have
\begin{equation}\label{E3.43}
0\in\Gamma(u).
\end{equation}
For any  $r\in (0,1)$, it is clear that $u_r\in\mathcal{G}_c$. Then we have
\begin{equation}\label{E3.61}
0\in\Gamma(u)\cap\Gamma(u_r).
\end{equation}
By (\ref{E3.62}), we have $\mathrm{Int}\,\Omega^{c}(u_r)\neq\emptyset.$ Applying (\ref{E3.42}), it is not hard to realize that
\begin{equation}\label{E3.45}
u_r(x)\rightarrow\frac{1}{2}x^TAx+\frac{D}{r^2}\log|P x|+\frac{C+D\log r}{r^2},\quad\text{as}\quad|x|\rightarrow +\infty.
\end{equation}

  Using $D<0$ again, we may choose sufficiently large $R>1$, such that $u_r< u$ on $\partial B_R.$ Applying Lemma \ref{lemma 3.2}, we obtain $\Omega^{c}(u_r)\subseteq\mathrm{Int}\,\Omega^{c}(u)$. This contradicts (\ref{E3.61}).

Next we show that (ii) holds. If $D=0$ , then $C\leq 0$ comes from Lemma \ref{lemma 3.1}. We need to show that $C=0$. Let us assume by contradiction that $C<0$. We claim that
\begin{equation}\label{E3.72}
|\Omega^{c}(u)|>0.
\end{equation}

Otherwise, we deduce that $u$ is a quadratic polynomial. Lemma \ref{lemma 3.1} implies that
\begin{equation*}
u=\frac{1}{2}x^TAx+C.
\end{equation*}
By $C\leq 0$ and $u\geq 0$, we have $C=0$, a contradiction. Thus (\ref{E3.72}) holds.

It suffices to consider the case when
\begin{equation}\label{E3.73}
F_{ij}(A)=\delta_{ij},
\end{equation}
where $F_{ij}(\cdot)$ denotes the derivative of the operator $F$. Hence $P=I$ (see (\ref{E1.24})).

Let
\begin{equation*}
w(x):=u(x)-\frac{1}{2}x^TAx-C,
\end{equation*}
and
\begin{equation}\label{E3.11}
 G(M):=F(M+A)-F(A),\qquad \forall M\in\mathcal{S}_{2\times 2}.
\end{equation}
Then we have
\begin{equation*}
G(D^2 w)=0\qquad\mbox{in}\quad \Omega(u).
\end{equation*}
Using (\ref{E3.41}), we deduce that
\begin{equation*}
|\nabla w|=O(|x|^{-2})\qquad\text{as}\quad |x|\rightarrow +\infty.
\end{equation*}

By the convexity of $F$, we have
\begin{equation}\label{E3.74}
\triangle w \leq F(A+D^2 w)-F(A)=0\qquad\mbox{in}\quad \Omega(u).
\end{equation}
Using the divergence theorem, for some $R>1$ large enough, we have
\begin{equation*}
\begin{split}
0\geq\int_{B_R\cap\Omega(u)}\triangle w&=\int_{\partial B_R}\frac{\partial w}{\partial r}ds-\int_{\Gamma(u)}\frac{\partial w}{\partial\nu}ds\\
&=\int_{\partial B_R}O(|x|^{-2})ds-\int_{\Gamma(u)}\frac{\partial(u-\frac{1}{2}x^TAx-C)}{\partial\nu}ds\\
&=O(R^{-1})+(\mbox{trace}A)|\Omega^{c}(u)|,
\end{split}
\end{equation*}
where $\nu(x)$ denotes the unit outward normal of $\Omega^c(u).$ Taking $R\rightarrow +\infty,$ we deduce that $|\Omega^{c}(u)|=0$. This contradicts (\ref{E3.72}).
\end{proof}

\begin{re}
For $n\geq 3$, if $p(x)=\frac{1}{2}(x-x_0)^TA(x-x_0)+D\in\mathcal{Q}_c$, then we have $u=\Phi^{-1}(p)\in\mathcal{G}_c$.
Repeating the arguments in the previous lemma, we find that either $x_0\in\mathrm{Int}\,\Omega^c(u)$, or $\Omega^c(u)=\{x_0\}.$
 \end{re}

The following lemma is a direct consequence of comparison principles.

\begin{lemma}
 \label{le:2.1}
For $n=2$, suppose that $\mathcal{F}$ is a uniformly elliptic operator. Let $V_R(x)$ satisfy
\begin{equation}\label{E3.51}
\left\{\begin{split}
\mathcal{F}(D^2 V_R)=0\qquad &\text{in}\qquad B_R\backslash B_1,\\
V_R=-\log R  \qquad  &\text{on}\qquad \partial B_R,\\
 V_R\leq 0\qquad &\text{on}\qquad \partial B_1,
\end{split}\right.
\end{equation}
in the viscosity sense, where $R>10$. Then
\begin{equation}\label{E3.52}
V_R\leq -c_0\log R\qquad \mbox{on}\qquad \partial B_{R/2},
\end{equation}
where the universal constant $c_0$ depends only on the ellipticity constants $\lambda$ and $\Lambda$.
\end{lemma}

\begin{proof}
 It is clear that $\sup V_R\leq 0$ follows from the maximum principle. Let $\beta=\Lambda/\lambda\geq 1.$ We consider the following auxiliary function:
\begin{equation*}
\Psi_R(x):=\left((2R-1)^{-\beta}-|x-2R e|^{-\beta}\right)R^{\beta}\log R,
\end{equation*}
where $e\in \partial B_1$. Then we have
\begin{equation*}
\left\{\begin{split}
\mathcal{F}(D^2 \Psi_R)\leq\mathcal{M}^+(D^2\Psi_R)\leq 0=\mathcal{F}(D^2 V_R)\quad &\mbox{in}\quad B_R(0)\cap B_{2R-1}(2R e),\\
\Psi_R=0\geq V_R  \qquad  &\mbox{on}\qquad B_R(0)\cap\partial B_{2R-1}(2R e),\\
\Psi_R>-\log R>V_R \qquad &\mbox{on}\qquad \partial B_R(0)\cap B_{2R-1}(2R e).
\end{split}\right.
\end{equation*}
The comparison principle implies that
\begin{equation*}
\begin{split}
V_R(R e/2)&\leq \Psi_R(R e/2)\leq \left((2R-1)^{-\beta} -(3/2)R^{-\beta}\right)R^{\beta}\log R\\
&\leq -c_0(\beta)\log R,
\end{split}
\end{equation*}
as we wanted to see.
\end{proof}

\begin{proof}[\textit{Proof of Theorem \ref{thm:5}}]
We divide the proof into three steps.

\indent\textit{Step 1. Definition of the map $\Phi$}

When the spatial dimension $n\geq 3$, the bijection $\Phi:\mathcal{G}_c\rightarrow\mathcal{Q}_c$ was established by Eberle-Yu
\cite{EY}, it remains to consider $n=2$. Let $u\in\mathcal{G}_c$ in $\mathbb{R}^{2}$. We define $\Phi(u):=q(x),$ where $q(x)\in\mathcal{Q}_c$ is given by Lemma \ref{lemma 3.1}.

\indent\textit{Step 2. Injection of the map $\Phi$}

Assuming $u, v\in\mathcal{G}_c$ with $q(x)=\Phi(u)=\Phi(v),$ we need to show that $u\equiv v$. By a translation of the coordinate system, we assume that
\begin{equation*}
q(x)=\frac{1}{2}x^TAx+D\log|Px|.
\end{equation*}

Indeed, if $D=0$, then $u=q=v$ follows from Lemma \ref{lemma 3.3}. It remains to consider the case that $D<0.$ Applying Lemma \ref{lemma 3.1}, we have
\begin{equation}\label{E3.63}
\left\{\begin{split}
u(x)=q(x)+C_1+O(|x|^{-1}),\\
v(x)=q(x)+C_2+O(|x|^{-1}),
\end{split}
\qquad\mbox{as}\quad |x|\rightarrow +\infty.
\right.
\end{equation}
where $C_1, C_2\in\mathbb{R}$. Applying Proposition \ref{pro:2.16} again, we have $u\equiv v$ provided $C_1=C_2$.

Next we show that $C_1=C_2$ in (\ref{E3.63}). On the contrary, we may assume that $C_1>C_2$. Choosing $R\gg1$, we have $u>v$ on $\partial B_R$. By Lemma \ref{lemma 3.3},
we have $\mathrm{Int}\,\Omega^c(u)\neq\emptyset$, and $\mathrm{Int}\,\Omega^c(v)\neq\emptyset.$ Applying Lemma \ref{lemma 3.2}, we have
$$
\Omega^{c}(u)\subseteq\mathrm{Int}\,\Omega^{c}(v).
$$

Denoting
\begin{equation*}
\delta:=\mbox{dist}( \Gamma(v), \Omega^c(u))>0,\qquad d:=\mbox{diam} (\Omega^c(u))>0,
\end{equation*}
we consider a global solution $u_r\in\mathcal{G}_c,$ where $r=1/(1+\frac{\delta}{2d})<1$. Thus
\begin{equation}\label{E3.68}
\Omega^{c}(u_r)=\left(1+\frac{\delta}{2d}\right) \Omega^{c}(u)\subseteq\mathrm{Int}\,\Omega^{c}(v).
\end{equation}

On the other hand, combining (\ref{E3.45}) and (\ref{E3.63}), we may choose $R>1$ large enough such that $u_r<v$ on $\partial B_R$.
Proposition \ref{pro:2.16} implies that $\Omega^{c}(v)\subseteq\Omega^{c}(u_r).$ This contradicts (\ref{E3.68}).

\indent\textit{Step 3. Surjection of the map $\Phi$}

Let $q(x)\in \mathcal{Q}_c$. We want to construct a global solution $u\in\mathcal{G}_c$ such that
\begin{equation*}
\Phi(u)=q(x).
\end{equation*}
 By a rescaling of $u$, we may assume that $F$ satisfies (\ref{E3.73}), and $P=I$. Hence
 \begin{equation*}
q(x)=\frac{1}{2}x^TAx+D\log|x|.
\end{equation*}
 If the constant $D=0$, then our result holds since $\Phi(q)=q(x).$ It remains to consider the case that $D<0.$

\textit{Step 3.1}  Let
\begin{equation*}
q_1(x)=\frac{1}{2}x^TAx-\log|x|.
\end{equation*}

For each $R>1$, let $u_R$ satisfy the following problem
\begin{equation}\label{E3.84}
\left\{\begin{split}
F(D^2 u_R)\leq 1\quad &\mbox{in}\quad B_R,\\
u_R\geq 0\quad&\mbox{in}\quad B_R,\\
F(D^2 u_R)= 1\quad &\mbox{in}\quad \{u_R>0\},\\
u_R=q_1\quad &\mbox{on}\quad \partial B_R,
\end{split}\right.
\end{equation}

We claim that
\begin{equation}\label{E3.85}
\left\{
\begin{split}
(q_1(x)-1/\lambda)^+ \leq u_R(x)\qquad &\mbox{in}\quad B_R\backslash B_1.\\
0\leq u_R(x)\leq \frac{1}{2}x^TAx \qquad &\mbox{in}\quad B_R.
\end{split}\right.
\end{equation}
In fact, the second inequality follows from Proposition \ref{pro:2.16}. It remains to prove the first inequality.
By convexity of $F$ and (\ref{E3.73}), we deduce that for any $x\neq 0,$
\begin{equation*}
\begin{split}
F(D^2q_1(x))-1&= F\big(A-D^2\log |x|\big)-F(A)\\
&\geq \mbox{trace}\Big(DF(A)\cdot D^2(-\log |x|)\Big)=\triangle(-\log |x|)=0.
\end{split}
\end{equation*}
Then we have
\begin{equation*}\label{E3.87}
\left\{\begin{split}
F(D^2q_1)\geq 1\geq F(D^2 u_R)\quad &\mbox{in}\quad B_R\backslash B_1,\\
q_1-1/\lambda \leq 0 \leq u_R \quad&\mbox{on}\quad \partial B_1,\\
u_R=q_1\quad &\mbox{on}\quad \partial B_R,
\end{split}\right.
\end{equation*}
where we are using $A\leq I_n/\lambda$. Therefore, the first inequality in (\ref{E3.85}) follows.

Together with $A>0$ and (\ref{E3.85}), there exists a universal constant $r_0>0$ such that $\Omega^c(u_R)\subseteq B_{r_0}$ for all $R\gg1$. By a rescaling, we assume that
\begin{equation*}
\Omega^{c}(u_R)\subseteq B_1,\quad\mbox{for all}\quad R\gg1.
\end{equation*}

It is not hard to observe that $\Omega^c(u_R)\neq\emptyset$ for all sufficiently large $R$. In fact, if $\Omega^c(u_R)=\emptyset$, then by (\ref{E3.84}) we deduce that
$u_R(x)=\frac{1}{2}x^TAx-\log R.$ Then we have $u_R(0)<0$, a contradiction.

For all sufficiently large $R$, denoting $v_R(x):=R^{-2}u_R(Rx)$, it is obvious that $v_R$ satisfies the obstacle problem in $B_1$ with
\begin{equation*}
\emptyset\neq\Omega^c(v_R)\subseteq B_{1/4}.
\end{equation*}
Applying the interior $C^{1,1}$ estimate (Proposition \ref{pro:2.11}), we have
\begin{equation*}
||D^2 u_R||_{L^{\infty}(B_{R/2})}=||D^2 v_R||_{L^{\infty}(B_{1/2})}\leq C(n,\lambda, \Lambda).
\end{equation*}
By (\ref{E3.85}), up to a subsequence, we have
\begin{equation}\label{E3.91}
u_R(x)\rightarrow u_\infty(x)\qquad\mbox{in}\quad C^{1,\alpha}_{loc}(\mathbb{R}^{2}),
\end{equation}
as $R\rightarrow +\infty.$ By (\ref{E3.85}), it is easily seen that
\begin{equation}\label{E3.92}
(q_1(x)-1/\lambda)^+ \leq u_\infty(x)\leq \frac{1}{2}x^TAx \qquad \mbox{in}\quad\mathbb{R}^{2}\backslash B_1.
\end{equation}

Thanks to \cite[Proposition 2.9]{Ca02}, we deduce that $u_\infty\in C^{1,1}_{loc}(\mathbb{R}^{2}) $ is a global solution
to the obstacle problem (\ref{E1.1}). Combining $A>0$ and (\ref{E3.92}), we have
\begin{equation}\label{E3.93}
u_\infty\in\mathcal{G}_c.
\end{equation}
Since $F$ satisfies (\ref{E3.73}), applying Lemma \ref{lemma 3.1}, we deduce that
\begin{equation}\label{E3.94}
u_\infty(x)=\frac{1}{2}(x-x_0)^T A_\infty (x-x_0)+D_\infty\log|x-x_0|+C_\infty+O(|x|^{-1})\qquad\mbox{as}\quad |x|\rightarrow +\infty.
\end{equation}
By (\ref{E3.92}), we observe that $A_\infty=A, x_0=0$ and $D_\infty\leq 0$.

\indent\textit{Step 3.2} We claim that $D_\infty<0$. Let us assume by contradiction that $D_\infty=0$. Applying Lemma \ref{lemma 3.3}, we have
 $$
 u_\infty=\Phi(u_\infty)=\frac{1}{2}x^T A x.
 $$
For any $\rho>1$, using (\ref{E3.91}), we have
\begin{equation}\label{E3.95}
u_R(x)-\frac{1}{2}x^T A x=u_R-u_\infty\rightarrow 0, \qquad\mbox{in}\quad C^{1,\alpha}(B_{2\rho}),
\end{equation}
as $R\rightarrow +\infty.$

 Fix $\epsilon>0$, and let
\begin{equation*}
v_R(x):=u_R(x)-\frac{1}{2}x^T A x-\epsilon.
\end{equation*}
Taking $R\gg\rho\gg1$ large enough, then we have
\begin{equation*}\label{E3.96}
\left\{\begin{split}
G(D^2 v_R)=0\qquad &\mbox{in}\qquad B_R\backslash B_\rho,\\
v_R\leq -\log R  \qquad  &\mbox{on}\qquad \partial B_R,\\
v_R \leq 0\qquad &\mbox{on}\qquad \partial B_\rho,
\end{split}\right.
\end{equation*}
where the convex operator $G$ is given by (\ref{E3.11}). Applying Lemma \ref{le:2.1}, we observe that
\begin{equation}\label{E3.98}
v_R\leq -c_0\log R\qquad \mbox{on}\qquad \partial B_{R/2},
\end{equation}
where $c_0$ is a universal constant. Define now
\begin{equation*}
U:=\left(B_{\frac{2}{3}R}\backslash B_\rho\right)\subset \Omega(u_R), \qquad  f_{R,\rho}(x):=\chi(U)(\triangle v_R(x)).
\end{equation*}
By (\ref{E3.85}), using the interior $C^{2}$ estimate again, we deduce that
\begin{equation}\label{E3.97}
|D^2 v_R|(x)\leq C(n,\lambda, \Lambda)(\log |x|)/|x|^2 \qquad\mbox{in}\qquad U.
\end{equation}

We compute
\begin{equation*}\label{E3.35}
\tilde{a}^{ij}(x)\partial_{ij}v_R(x)=F(A+D^2 v_R)-F(A)=0,\quad\mbox{in}\quad \mathbb{R}^{2}\backslash B_\rho,
\end{equation*}
where $\tilde{a}^{ij}(x):=\int\limits_{0}^{1}F_{ij}[A+tD^2 v_R(x)]dt.$

Since $F\in C^{1,\alpha_F}$, it follows that for some $C>0$,
\begin{equation}\label{E3.101}
|\tilde{a}^{ij}(x)-F_{ij}(A)|\leq C(|D^2 v_R|(x))^{\alpha_F}\leq C|x|^{-(5/3)\alpha_F}\quad\mbox{in}\quad U.
\end{equation}
Putting (\ref{E3.97}), (\ref{E3.101}), and (\ref{E3.73}) together,  we have
\begin{equation}\label{E3.102}
|f_{R,\rho}(x)|\leq|\delta_{ij}-\tilde{a}^{ij}(x)||\partial_{ij}v_R(x)|
\leq C|x|^{-2-(4/3)\alpha_F}\qquad\mbox{in}\qquad U.
\end{equation}

Let
\begin{equation*}
w_{R,\rho}(x):=\frac{1}{2\pi}\int\limits_{U}\log|x-y|f_{R,\rho}(y)dy,\quad x\in\mathbb{R}^{2}.
\end{equation*}
We claim that
\begin{equation}\label{E3.103}
\left\{\begin{split}
\triangle (v_R-w_{R,\rho})(x)=0 \qquad &\mbox{in}\qquad B_{R/2}\backslash B_\rho,\\
v_R-w_{R,\rho}\leq -\frac{c_{0}}{2}\log R  \qquad  &\mbox{on}\qquad \partial B_{R/2},\\
v_R-w_{R,\rho}\leq \rho^{-\alpha_F}\qquad &\mbox{on}\qquad \partial B_\rho,
\end{split}\right.
\end{equation}
 where the constant $c_0$ is given by (\ref{E3.98}).

Indeed, (\ref{E3.103}) can be seen as follows. For any $e\in \partial B_1$, we can compute separately
\begin{equation}\label{E3.104}
 |w_{R,\rho}(Re/2)|\leq \left(\int\limits_{U\backslash B_1(Re/2)}+\int\limits_{B_1(Re/2)\cap U}\right)|f(y)|\log|Re/2-y|dy:=I_1+I_2.
\end{equation}
 By (\ref{E3.102}), we have
\begin{equation*}
I_1\leq C\log R\int\limits_{U}|y|^{-2-\alpha_F}dy\leq C\rho^{-\alpha_F}\log R,
\end{equation*}
 and
\begin{equation*}
I_2\leq CR^{-2-\alpha_F}\int\limits_{B_1(Re/2)}\log|Re/2-y|dy\leq R^{-2}.
\end{equation*}
 Taking $\rho>1$ large enough, then the second inequality in (\ref{E3.103}) comes from (\ref{E3.98}) and (\ref{E3.104}). Next we can compute
\begin{equation}\label{E3.105}
 |w_{R,\rho}(\rho e)|\leq \left(\int\limits_{U\backslash B_1(\rho e)}+\int\limits_{U\cap B_1(\rho e)}\right)|f(y)|\log|\rho e-y|dy:=J_1+J_2.
\end{equation}
It is easily seen that
\begin{equation*}
J_1\leq C\int\limits_{U}(\log \rho+\log |y|)|y|^{-2-(4/3)\alpha_F}dy\leq\rho^{-\alpha_F}/2,
\end{equation*}
and
\begin{equation*}
J_2\leq C\rho^{-2-(4/3)\alpha_F}\int\limits_{B_1(\rho e)}\log|\rho e-y|dy\leq \rho^{-2},
\end{equation*}
by taking $\rho$ sufficiently large. Since $v_R\leq 0$ in $B_{2\rho}\backslash B_{\rho}$ (recalling (\ref{E3.95})), we deduce the third inequality in (\ref{E3.103})
by taking $R\gg\rho$ large enough if necessary.

We now define a harmonic function $h$ by
\begin{equation*}
h(x):=\rho^{-\alpha_F}-\frac{c_{0}}{2}\log (|x|/\rho).
\end{equation*}
By (\ref{E3.103}), we deduce that
\begin{equation*}\label{E3.110}
v_R-w_{R,\rho}\leq h, \qquad \mbox{in}\qquad B_{R/2}\backslash B_\rho,
\end{equation*}
where we are using the comparison principle. Hence
\begin{equation}\label{E3.107}
v_R-w_{R,\rho}\leq h=\rho^{-\alpha_F}-\frac{c_{0}}{2}\log 2,\qquad \mbox{on}\qquad \partial B_{2\rho}.
\end{equation}

On the other hand, repeating the preceding computation, we obtain $|w_{R,\rho}|\leq\rho^{-\alpha_F}$ on $\partial B_{2\rho}$. In summary, we find that
\begin{equation*}
|v_R-w_{R,\rho}|\leq\epsilon+O(\rho^{-\alpha_F})\qquad \mbox{on}\qquad \partial B_{2\rho}.
\end{equation*}
This contradicts (\ref{E3.107}) upon taking $\epsilon\rightarrow 0^+$ (recalling $\rho\gg1$). Therefore $D_\infty<0$.

\indent\textit{Step 3.3} Let
\begin{equation*}
u(x):=\frac{D}{D_\infty}u_\infty\left(\sqrt{\frac{D_\infty}{D}}x\right).
\end{equation*}
It is easily seen that $\Phi(u)=q(x).$

The proof is completed.
\end{proof}

\section{ The coincidence set is slender provided $u\in\mathcal{G}_{\infty}$. }

  To simplify our statement, let us introduce a few notations. Suppose that $\mathcal{K}$ is a closed convex set in $\mathbb{R}^{n}(n\geq 2)$ with $\mathrm{Int}\,\mathcal{K}\neq\emptyset$. Suppose further that $\partial\mathcal{K}$ has local $C^1$ regularity. For any  $x_0\in\partial\mathcal{K}$, we denote by $\nu(x_0)$ the unit outward normal of $\mathcal{K}$
  at $x_0$, and we denote by $L_{x_0}$ the supporting hyperplane of $\mathcal{K}$ at $x_0$. Given $d>0$, the set $\Omega_d(\mathcal{K})$ is defined as follows:
 \begin{equation*}
\Omega_d(\mathcal{K}):=\{x\in\mathbb{R}^{n}: \mbox{dist}(x,\mathcal{K})>d\}.
\end{equation*}

Fixing $r\geq 1$, the cylinder $\mathcal{C}_t$ is defined as follows:
 \begin{equation*}
 \mathcal{C}_t:=\mathcal{C}_{rt,t}(0)=B'_{rt}(0)\times (-t,t),\qquad\mbox{for all}\quad t\in\mathbb{R}_{+}.
 \end{equation*}

 The following lemma on linear equations is useful for our proof. Similar ideas were used in the work by Silva-Savin \cite{SS}.

\begin{lemma}
 \label{lemma 4.0}
Suppose that the origin $0\in\partial\mathcal{K}$. Suppose further that there exists a unit vector $e\in\partial B'_1$ such that
\begin{equation}\label{E4.74}
\theta:=\angle\vartheta(e,\nu(0))<\pi/2.
\end{equation}
Let the linear elliptic operator $\mathcal{L}$ be given by (\ref{E2.45}), and let $r\geq 1$.

Then there are universal constants $M(n,\lambda,\Lambda)>1$ large, and $\delta(n,\lambda,\Lambda,\theta)>0$ small, such that, if $w$ satisfies
\begin{equation*}
\left\{\begin{split}
\mathcal{L}w&=0 \quad\mbox{in}\quad\mathcal{C}_1,\\
w&=0\quad\mbox{in}\quad\mathcal{C}_1\cap\mathcal{K},\\
w&>M\quad\mbox{in}\quad\mathcal{C}_1\cap\Omega_{\delta}(\mathcal{K}),\\
w&\geq -1\quad\mbox{in}\quad\mathcal{C}_1,
\end{split}\right.
\end{equation*}
 then
\begin{equation}\label{E4.70}
w(te)\geq 0,\qquad\mbox{for all}\quad t\in(0,r).
\end{equation}
\end{lemma}

\begin{proof}
We divide the proof into three steps.

\indent\textit{Step 1. }
 We show that there exists a constant $\alpha\in(0,1)$, depending only on $n,\lambda,\Lambda$, such that
 \begin{equation}\label{E4.71}
 w>\alpha M,\qquad\mbox{in}\quad\mathcal{C}_{1-3\delta}\cap\Omega_{\delta/2}(\mathcal{K}).
 \end{equation}

Let the constant $\delta\in(0,1/32)$ be chosen later. Picking $x\in\mathcal{C}_{1-3\delta}\cap\big(\Omega_{\delta/2}(\mathcal{K})\backslash\Omega_{\delta}(\mathcal{K})\big)$,
we can select a point $x_0\in\partial\mathcal{K}$, such that $|x-x_0|=\dist(x,\mathcal{K})$. Denoting $d_x:=|x-x_0|$, then $d_x\in(\delta/2,\delta)$. Let $z:=x+d_x\nu(x_0)$.
Then we have
\begin{equation*}
z\in\mathcal{C}_{1-2\delta},\qquad\dist(z,\mathcal{K})\geq\dist(z,L_{x_0})=2d_x.
\end{equation*}

By
\begin{equation*}
\left\{\begin{split}
\mathcal{L}(w+1)&=0,\\
w+1&\geq 0,
\end{split}\right.
\quad\mbox{in}\quad B_{2d_x}(z)\subset\mathcal{C}_1,
\end{equation*}
we deduce that
\begin{equation*}
w(x)+1\geq\inf_{B_{d_x}(z)}(w+1)\geq c_1(w(z)+1)\geq c_1(M+1),
\end{equation*}
where $c_1$ is a constant depending only on $n,\lambda,\Lambda$.

Taking $M=2/c_1$ and $\alpha=c_1/2$, we have
 \begin{equation*}
w(x)>1=\alpha M.
\end{equation*}

\indent\textit{Step 2. }
 We show that
 \begin{equation}\label{E4.72}
 \left\{\begin{split}
 w\geq-\alpha,\qquad&\mbox{in}\quad\mathcal{C}_{1/2},\\
 w(te)>0,\qquad&\mbox{for}\quad t\in[r/2,r).
 \end{split}\right.
 \end{equation}

To verify (\ref{E4.72}), we claim that
\begin{equation}\label{E4.75}
V_0\left(e,\frac{\pi}{4}-\frac{\theta}{2}\right)\cap(\mathcal{K}\backslash\{0\})=\emptyset,
\end{equation}
(see Notations (8)). In fact, given $0\neq x\in V_0\left(e,\frac{\pi}{4}-\frac{\theta}{2}\right)$, by (\ref{E4.74}), we find that
\begin{equation*}
\angle\vartheta(x,\nu(0))\leq\angle\vartheta(\nu(0),e)+\left(\frac{\pi}{4}-\frac{\theta}{2}\right)<\frac{\pi}{2}.
\end{equation*}
By the convexity of $\mathcal{K}$, $L_0$ separates $\mathcal{K}$ and $V_0\left(e,\frac{\pi}{4}-\frac{\theta}{2}\right)$ (Recall that $L_0$ is the supporting hyperplane at the origin). Hence (\ref{E4.75}) follows.

 Let $\delta\in\big(0,\sin(\pi/4-\theta/2)\big)$ be a small constant to be chosen later. By (\ref{E4.75}) and $r\geq1$, we observe that
\begin{equation*}
\dist(te,\mathcal{K})>\frac{r}{2}\sin(\pi/4-\theta/2)\geq\delta/2\quad\mbox{for}\quad t\in[r/2,r).
\end{equation*}
Then we obtain that $w(te)>0$ for $t\in[r/2,r)$.

Next we show that $w\geq-\alpha$ in $\mathcal{C}_{1/2}$. Picking $x\in\mathcal{C}_{1-3\delta}\cap\Big(\Omega_{\delta/2}(\mathcal{K})\backslash\Omega_{\delta}(\mathcal{K})\Big)$,
let $x_0$ be defined as in the previous step. Let $z:=x+\delta\nu(x_0)$. Then $z\in\mathcal{C}_{1-2\delta}$. By
\begin{equation*}
B_{\delta}(z)\cap\mathcal{K}=\emptyset,
\end{equation*}
we have
\begin{equation*}
\inf_{y\in B_{\delta/2}(z)}\dist(y,\mathcal{K})\geq\inf_{y\in B_{\delta/2}(z)}\dist(y,L(x_0))>\delta/2.
\end{equation*}
Using (\ref{E4.71}), we have $w>0$ in $B_{\delta/2}(z).$ Then $w^{-}$ satisfies
\begin{equation*}
\left\{\begin{split}
\mathcal{L}w^{-}&\geq 0,\quad\mbox{in}\quad B_{2\delta}(z)\subset\mathcal{C}_1,\\
|\{B_{2\delta}(z)&:w^{-}=0\}|\geq|B_{\delta/2}(z)|.
\end{split}\right.
\end{equation*}

Applying the weak Harnack inequality for subsolution (cf. \cite{KS}), we deduce that
\begin{equation*}
w^{-}(x)\leq\sup_{B_{\delta}(z)}w^{-}\leq(1-\beta)\sup_{B_{2\delta}(z)}w^{-}\leq 1-\beta,
\end{equation*}
where $\beta>0$ is a constant depending only on $n,\lambda,\Lambda$, in other words, we have
\begin{equation*}
w^{-}\leq 1-\beta\qquad\mbox{in}\quad\mathcal{C}_{1-3\delta}.
\end{equation*}

Iterating this result, we find that
\begin{equation*}
w\geq-(1-\beta)^{k}\qquad\mbox{in}\quad\mathcal{C}_{1-3k\delta}.
\end{equation*}
Then (\ref{E4.72}) follows by choosing
\begin{equation*}
k=\left[\frac{\log\alpha}{\log(1-\beta)}\right]+1,\quad\quad\delta=\min\big((6k)^{-1}, \sin(\pi/4-\theta/2), 1/32\big).
\end{equation*}

\indent\textit{Step 3. } After a dilation, we repeat the process again. Consider
\begin{equation*}
w_{j}(x):=\alpha^{-j}w(x/2^j),\qquad \mathcal{K}_{j}:=\{x:x/2^j\in\mathcal{K}\},\qquad j=1,2,\cdots
\end{equation*}
By iterating, we conclude that
\begin{equation*}
 w\geq-\alpha^{-j}\quad\mbox{in}\quad\mathcal{C}_{2^{-j}},\qquad\mbox{and}\qquad w(te)>0\quad\mbox{for}\quad t\geq 2^{-j}r.
 \end{equation*}
Hence we obtain (\ref{E4.70}).
\end{proof}

\begin{defn}
\label{def: vector of unbounded }
Let $u(x)$ be a nonnegative continuous function in $\mathbb{R}^{n}$. We define
\begin{equation*}
V(u):=\left\{e\in\partial B_1:\mbox{there exists} \{P_k\}_{k=1}^{\infty}\in\Omega^{c}(u), \mbox{such that}\\ |P_k|\rightarrow +\infty, P_k/|P_k|\rightarrow e.\right\}
\end{equation*}
\end{defn}

\begin{lemma}
 \label{le:2.4}
Let $u(x)$ be a nonnegative convex function in $\mathbb{R}^{n}(n\geq 2).$ If $e\in V(u)$, then $u(\cdot+te)\leq u(\cdot)$ for $t\in\mathbb{R}_+.$ Clearly, we have
$\partial_e u\leq 0$ if $u\in C^1(\mathbb{R}^{n})$.
\end{lemma}

\begin{proof}
Without loss of generality, we set $e=e_{n}$. In other words, there exists a sequence $P_k(P_k',(P_k)_n)$ such that $(P_k)_n\rightarrow+\infty,P_k'/(P_k)_n\rightarrow 0,$
and $u(P_k)=0.$

For any $t\in\mathbb{R}_+$, choosing $\varepsilon_k=t/(P_k)_n\rightarrow 0^+$, we have $|\varepsilon_k P_k'|=o(1)$. By convexity of $u$, it is easily seen that
\begin{equation*}
 \begin{split}
u[(1-\varepsilon_k)x'+\varepsilon_k P_k', (1-\varepsilon_k)x_n+t]&=u[(1-\varepsilon_k)x+\varepsilon_k P_k]\\
&\leq(1-\varepsilon_k)u(x',x_n)+\varepsilon_k u(P_k).
 \end{split}
\end{equation*}
Hence the result follows by taking $k\rightarrow+\infty$.
\end{proof}

\begin{lemma}
 \label{le:2.6}
Let $u(x)$ be a nonnegative convex function in $\mathbb{R}^{n}(n\geq 2)$. Suppose that $\mathrm{Int}\,\Omega^{c}(u)\neq\emptyset$. Suppose further that
$V(u)\neq\emptyset$. Then $\mathbb{R}_+V(u)$ is a closed convex cone in $\mathbb{R}^{n}$.
\end{lemma}

\begin{proof}
Without loss of generality, we assume in addition that the origin $0\in\Omega^{c}(u)$. For any unit vector $\tau\in V(u)$, Lemma \ref{le:2.4} implies that the ray
$\{\mathbb{R}_+\tau\}\subseteq\Omega^{c}(u).$ Hence $\mathbb{R}_+V(u)$ is a cone.

We verify that the cone is convex. Choosing $x_1,x_2\in\mathbb{R}_+V(u)$, let
$$
x_{\theta}:=\theta x_1+(1-\theta)x_2,\qquad \forall\theta\in (0,1).
$$
We want to show $x_{\theta}\in\mathbb{R}_+V(u)$. Once we verify that $sx_{\theta}\in\Omega^{c}(u)$ for any $s\geq 0$, this result follows. In fact, Lemma \ref{le:2.4} implies that
$sx_{1},sx_{2}\in\Omega^{c}(u)$. Since $u\geq 0$ in $\mathbb{R}^{n}$, $u(sx_{\theta})=0$ follows from the convexity of $u$.

Assume now that we have unit vectors $\{\nu_k\}_{k=1}^{\infty}\subset V(u)$ with $\nu_k\rightarrow\nu$. Then there exists
$\{P^{k}_{j}\}_{j=1}^{\infty}\subset\Omega^{c}(u)$ such that
$$
P^{k}_{j}/|P^{k}_{j}|\rightarrow \nu_k \qquad\mbox{as}\quad j\rightarrow +\infty.
$$
Taking a subsequence if necessary, we have
$$
P^{k}_{j}/|P^{k}_{j}|\rightarrow\nu\qquad \mbox{as}\quad j, k\rightarrow +\infty.
$$
Then we deduce that $\nu\in V(u)$, that is, the cone $\mathbb{R}_+V(u)$ is closed.
\end{proof}

\begin{cor}
 \label{cor: 4.1}
 Suppose that all hypotheses of Theorem \ref{thm:1} hold with the spatial dimension $n\geq 2$. Assume that $V(u)\neq\emptyset$. Assume in addition that $u$ is \textbf{not}
 a lower dimensional global solution. Then we have $\mathrm{Int}\,\Omega^{c}(u)\neq\emptyset$, and
\begin{equation}\label{E4.8}
\sup_{e,\tau\in V(u)}\angle\vartheta(e,\tau)<\pi.
\end{equation}
 \end{cor}

\begin{proof}
Assuming by contradiction that $\mathrm{Int}\,\Omega^{c}(u)=\emptyset$, we have proved $u(x)=\frac{1}{2}x^TAx$ in the previous section. Then $\det(A)=0$ comes from $V(u)\neq\emptyset$, i.e., $u$ is a lower dimensional global solution. This contradicts one of our assumptions.

Next we show that (\ref{E4.8}) holds. On the contrary, we suppose that (\ref{E4.8}) fails. By Proposition \ref{pro:2.12}, we see that $u$ is convex.
Since $\mathbb{R}_+V(u)$ is a closed cone (see Lemma \ref{le:2.6}), we may choose unit vectors $\pm e\in V(u)$. By Lemma \ref{le:2.4}, we have $\pm(\partial_e u)\leq 0$.
Hence $u$ is a lower dimensional function, a contradiction.
\end{proof}

\begin{re}
\label{re:4.1}
Suppose that all hypotheses of Corollary \ref{cor: 4.1} hold. After a possible translation and rotation of the coordinate system, we can assume that
\begin{equation}\label{E6.3}
0\in\Omega^{c}(u)\subseteq\{x\in\mathbb{R}^n : x_n\leq 0\},\qquad -e_n\in V(u),
\end{equation}
and
\begin{equation}\label{E3.32}
\sup_{\tau\in V(u)}\angle\vartheta(-e_n,\tau)=\theta_0<\pi/2.
\end{equation}
Let $l=\tan\theta_0+1$. By rescaling, we observe that there exists $R_0>1$ large, such that
\begin{equation}\label{E4.11}
\Omega^{c}_{1}(u_R)\subseteq B'_{l}(-e_n),\quad\mbox{for all}\quad R>R_0,
\end{equation}
where $\Omega^{c}_{t}(u_R):=\Omega^{c}(u_R)\cap\{x_n=-t\}.$
\end{re}

\begin{lemma}
 \label{lemma 4.5}
Suppose that all hypotheses of Corollary \ref{cor: 4.1} hold. Suppose further that (\ref{E6.3})-(\ref{E4.11}) hold. Then there exists $\theta_1(\theta_0)\in [0,\pi/2)$, such that
\begin{equation*}
\angle\vartheta(\nu(x_0),\nu'(x_0))\leq\theta_1<\pi/2,\qquad\mbox{for all}\quad x_0\in\partial\Omega^{c}_1(u_R),
\end{equation*}
where $\nu'(x_0)\in\partial B'_1$ denotes the unit outward normal of $\partial\Omega^{c}_1(u_R)$.
\end{lemma}

\begin{proof}
  By Lemma \ref{le:2.4}, $\partial_n u\geq 0$ follows from $-e_n\in V(u).$  Fixing $R>R_0$, we have
\begin{equation}\label{E4.10}
\{-\mathbb{R}_+e_n\}\subseteq\Omega^{c}(u_R).
\end{equation}
 Notice that $L_{x_0}$ is the supporting hyperplane of $\Omega^{c}(u_R)$.

 If $\angle\vartheta(\nu(x_0),e_n)=0$ (or $\pi$), then $L_{x_0}(x)=\{x: x_n=-1\}$.
 This contradicts the fact that $\Omega^{c}(u_R)$ lies on the same side of $L_{x_0}$. If $\angle\vartheta(\nu(x_0),e_n)\in(\pi/2,\pi)$, then $\{-\mathbb{R}_+e_n\}\cap\Omega^{c}(u_R)=\emptyset$, a contradiction. In summary, we deduce that
\begin{equation}\label{E4.76}
0<\angle\vartheta(\nu(x_0),e_n)\leq\pi/2,\quad\mbox{for all}\quad x_0(x'_0,-1)\in\partial\Omega^{c}_1(u_R).
\end{equation}

Denoting by $\tau$ the projection of $\nu(x_0)$ onto the hyperplane $\{x_n=-1\}$, we claim that $\nu'(x_0)=\tau/|\tau|$.

In fact, $\Omega^{c}_1(u_R)$ is a convex set, and $L_{x_0}\cap\{x_n=-1\}$ is its supporting hyperplane. Hence $\tau\bot(L_{x_0}\cap\{x_n=-1\})$. By (\ref{E4.76}), we observe that $0\leq \angle\vartheta(\nu(x_0),\tau)<\pi/2.$
Then we find that $\{\mathbb{R}_+\tau\}\cap\Omega^{c}(u_R)=\emptyset$, that is, $\angle\vartheta(\nu'(x_0),\tau)=0$. The claim follows.

Using (\ref{E4.76}) again, we have
\begin{equation*}
\nu(x_0)=a\nu'(x_0)+\sqrt{1-a^2}e_n,
\end{equation*}
where $a\in(0,1]$. For the case $a=1$, our conclusion is trivial.

Next we consider the case $a\in(0,1).$ By a possible rotation of the coordinate system, we assume that $\nu'(x_0)=e_1$. Then we have
\begin{equation*}
L_{x_0}(x)=\{x: ax_1+\sqrt{1-a^2}x_n+\sqrt{1-a^2}-ax_{0,1}=0\},
\end{equation*}
where $x_{0,1}:=x_0\cdot e_1.$ One can check that
\begin{equation*}
\bigl(ax_{0,1}-\sqrt{1-a^2}\bigr)\frac{e_n}{\sqrt{1-a^2}}\in L_{x_0}.
\end{equation*}

Since $\Omega^{c}(u_R)$ lies on the same side of $L_{x_0}$, using (\ref{E4.10}), we find that
\begin{equation*}
ax_{0,1}-\sqrt{1-a^2}\geq 0.
\end{equation*}
By (\ref{E4.11}), we have $x_{0,1}<l$. Thus we deduce that
\begin{equation*}
\cos\angle\vartheta(\nu(x_0),\nu'(x_0))=a\geq(1+x^2_{0,1})^{-1/2}>(1+l^2)^{-1/2}.
\end{equation*}
Then the result follows.
\end{proof}

\begin{thm}
 \label{thm:thm4.1}
 Suppose that all hypotheses of Corollary \ref{cor: 4.1} hold. Then $V(u)=\{\tau\}$, where $\tau$ is a unit vector.
\end{thm}

\begin{proof}
We divide the proof into four steps.

\indent\textit{Step 1. }
Suppose on the contrary that $\tau_1,\tau_2\in V(u)\cap\partial B_1$, where $\tau_1\neq\tau_2$. Applying Corollary \ref{cor: 4.1}, we have
 \begin{equation*}
0<\sup_{e,\tau\in V(u)}\angle\vartheta(e,\tau)=2\theta_0<\pi.
\end{equation*}
We now define
\begin{equation*}
\tau(\alpha):=\Big(\cos(\pi/2+\alpha),0,\cdots,0,\sin(\pi/2+\alpha)\Big).
\end{equation*}

 Without loss of generality, we can assume that $-\tau(-\theta_0),-\tau(\theta_0)\in V(u)$. Notice that $\tau(0)=e_n$. Since $u$ is a convex function,
by Lemma \ref{le:2.6}, we obtain
 \begin{equation*}
 -\tau(\alpha)\in V(u),\qquad\mbox{for all}\quad\alpha\in[-\theta_0,+\theta_0].
 \end{equation*}
Using Lemma \ref{le:2.4}, we observe that
\begin{equation*}
 \partial_{\tau(\alpha)}u\geq 0,\qquad\mbox{for all}\quad\alpha\in[-\theta_0,+\theta_0].
 \end{equation*}
Fixing $\alpha\in[-\theta_0,+\theta_0]$, then $\partial_{\tau(\alpha)}u$ satisfies
\begin{equation}\label{E6.10}
\left\{\begin{split}
a^{ij}(x)\partial_{ij}w=0 &\quad\mbox{in}\quad\Omega(u),\\
w=0 &\quad\mbox{on}\quad \Gamma(u),
\end{split}\right.
\end{equation}
where $a^{ij}(x)=F_{ij}(D^{2}u(x))$. Then we observe that
\begin{equation*}
\partial_{\tau(\alpha)}u>0\quad\mbox{in}\quad\Omega(u),\qquad\mbox{for}\quad\alpha\in[-\theta_0,+\theta_0].
\end{equation*}
Others, the strong maximum principle implies $\partial_{\tau(\alpha)}u\equiv 0$ in $\Omega(u)$, a contradiction.

By Remark \ref{re:4.1}, we may assume that (\ref{E6.3})-(\ref{E4.11}) holds. Let $T:=\{x_1,0,\cdots,0,x_n\}$. We denote by $D_{T}u_R(e_n)$ the projection of $Du_R(e_n)$ onto
the plane $T$. Let
\begin{equation*}
 \tau\big(\beta(e_n)\big):=D_{T}u_R(e_n)/|D_{T}u_R(e_n)|.
\end{equation*}
Since $\partial_{n}u_R(e_n)>0$, it is clear that $\angle\vartheta\Big(\tau\big(\beta(e_n)\big),e_n\Big)<\pi/2$, i.e.,
\begin{equation*}
\beta(e_n)\in(-\pi/2,\pi/2).
\end{equation*}

\indent\textit{Step 2. }
 We show that there exists a universal constant
$N>1$, depending only on $n,\lambda,\Lambda,\theta_0$ such that either
\begin{equation}\label{E4.80}
\partial_{n}u_R\leq N\partial_{\tau(-\theta_0)}u_R\quad\mbox{in}\quad B_1,
\end{equation}
or
\begin{equation}\label{E4.81}
\partial_{n}u_R\leq N\partial_{\tau(\theta_0)}u_R\quad\mbox{in}\quad B_1,
\end{equation}
for all $R>R_0$, where $R_0$ is given by (\ref{E4.11}).

\indent\textit{Step 2.1}
We verify that (\ref{E4.80}) holds provided $\beta(e_n)\in(-\pi/2,0]$.

 Let $\epsilon_R:=\partial_{n}u_R(e_n)>0$. Choosing $l=\tan\theta_0+1$, the interior Harnack inequality implies that
 \begin{equation*}
\partial_{n}u_R<C(n,\lambda,\Lambda,\theta_0)\epsilon_R\quad\mbox{on}\quad B'_{6l}(2e_n).
\end{equation*}
 Using the convexity of $u_R$, we find that
 \begin{equation}\label{E4.83}
\partial_{n}u_R<C_1(n,\lambda,\Lambda,\theta_0)\epsilon_R\quad\mbox{in}\quad\mathcal{C}_{6l, 2}(0).
\end{equation}
By $\beta(e_n)\in(-\pi/2,0]$, we observe that
 \begin{equation*}
\angle\vartheta\Big(\tau\big(\beta(e_n)\big),\tau(-\theta_0)\Big)\leq\max(\theta_0,\pi/2-\theta_0).
\end{equation*}
Denoting $\gamma:=\max(\theta_0,\pi/2-\theta_0)$, then we have
\begin{equation}\label{E4.85}
\partial_{\tau(-\theta_0)}u_R(e_n)\geq|D_{T}u_R(e_n)|\cos\gamma\geq\epsilon_R\cos\gamma.
\end{equation}

To prove (\ref{E4.80}),  we claim that
\begin{equation}\label{E4.84}
\partial_{n}u_R\leq N\partial_{\tau(-\theta_0)}u_R\quad\mbox{on}\quad\partial\mathcal{C}_{3l,1}(0).
\end{equation}
This claim can be seen as follows. Picking $y\in\partial\mathcal{C}_{3l,1}(0)$, since $|Du_R|=0$ on $\Omega^{c}(u_R)$, we may assume that $y\in\Omega(u_R)$.

For the case $y\in B'_{3l}(e_n)\cup(\partial B'_{3l}(0)\times(-1,1)),$ using (\ref{E3.32}), it is easily seen that $\Omega^{c}_2(u_R)\subseteq B'_{2l}(-2e_n)$ for all $R>R_0.$
Using $\partial_n u_R\geq 0$, we have
 \begin{equation*}
\Omega^{c}(u_R)\cap\{x_n>-2\}\subseteq B'_{2l}(0)\times[-2,0].
\end{equation*}
Hence dist$(y,\Omega^{c}(u_R))\geq 1$. Since $\partial_n u_R$ and $\partial_{\tau(-\theta_0)}u_R$ are solutions to (\ref{E6.10}), together with (\ref{E4.83})-(\ref{E4.85}), we obtain $\partial_{n}u_R(y)\leq N\partial_{\tau(-\theta_0)}u_R(y)$ provided $N>1$ sufficiently large, where we are using the Harnack inequality.

For the case $y(y',-1)\in B'_{3l}(-e_n)\cap\Omega(u_R)$, we can select $y_0(y'_0,-1)\in\partial\Omega^{c}_1(u_R)$ such that $|y-y_0|=\dist(y,\Omega^{c}_1(u_R))$.
Since $\{-e_n\}\in\mathrm{Int}\,\Omega^{c}_1(u_R)$, $|y-y_0|\leq 3l$ is implied. Using $|y_0|<l$ (see (\ref{E4.11})), we have
\begin{equation*}
y=y_0+|y-y_0|\nu'(y_0)\in\mathcal{C}_{3l,1}(y_0)\subset\mathcal{C}_{6l,2}(0).
\end{equation*}
(Recall that $\nu'(y_0)\in\partial B'_1$ denotes the unit outward normal of $\partial\Omega^{c}_1(u_R)$). Applying Lemma \ref{lemma 4.5}, we obtain
\begin{equation*}
\angle\vartheta(\nu(y_0),\nu'(y_0))\leq\theta_1(\theta_0)<\pi/2.
\end{equation*}

Let the universal constants $M(n,\lambda,\Lambda)>1$ and $\delta(n,\lambda,\Lambda,\theta_1)>0$ be given by Lemma \ref{lemma 4.0}. By (\ref{E4.85}), using
the Harnack inequality, we deduce that
\begin{equation}\label{E4.86}
\partial_{\tau(-\theta_0)}u_R\geq C(n,\lambda,\Lambda,\delta)\partial_{\tau(-\theta_0)}u_R(e_n)=C_2(n,\lambda,\Lambda,\theta_0)\epsilon_R\quad\mbox{in}\quad\Omega^*_{\delta},
\end{equation}
where the set $\Omega^*_{\delta}$ is defined as follows:
\begin{equation*}
\Omega^*_{\delta}:=\big\{x\in\mathcal{C}_{3l, 1}(y_0):\dist(x,\Omega^c(u_R))>\delta\big\}.
\end{equation*}

Let
\begin{equation*}
w:=\frac{M+2}{C_2\epsilon_R}\partial_{\tau(-\theta_0)}(u_R)-\frac{1}{C_1\epsilon_R}\partial_{n}u_R.
\end{equation*}
Then we have
\begin{equation*}
a^{ij}(x)\partial_{ij}w=0,\quad\mbox{in}\quad\mathcal{C}_{3l,1}(y_0)\backslash\Omega^c(u_R),
\end{equation*}
where $a^{ij}(x)=F_{ij}(D^{2}u_R(x))$. Combining (\ref{E4.83}) and (\ref{E4.86}), we find that
\begin{equation*}
w>M\quad\mbox{in}\quad\Omega^*_{\delta},\qquad w\geq -1\quad\mbox{in}\quad\mathcal{C}_{3l,1}(y_0).
\end{equation*}
Applying Lemma \ref{lemma 4.0}, we conclude that
\begin{equation*}
w(t\nu'(y_0)+y_0)\geq 0,\quad\mbox{for}\quad t\in(0,3l).
\end{equation*}
Hence
\begin{equation*}
\frac{M+2}{C_2}\partial_{\tau(-\theta_0)}(u_R)(y)\geq\frac{1}{C_1}\partial_{n}(u_R)(y).
\end{equation*}
Taking $N(n,\lambda,\Lambda,\theta_0)=C_1(M+2)/C_2$, we obtain (\ref{E4.84}).

 Consider $v:=N\partial_{\tau(-\theta_0)}u_R-\partial_{n}u_R$. Since
\begin{equation*}
\left\{\begin{split}
a^{ij}(x)\partial_{ij}v=0 &\quad\mbox{in}\quad\mathcal{C}_{3l,1}\cap\Omega(u_R),\\
v=0 &\quad\mbox{on}\quad \Gamma(u_R),\\
v\geq 0 &\quad\mbox{on}\quad\partial\mathcal{C}_{3l,1}\cap\Omega(u_R),\\
\end{split}\right.
\end{equation*}
 (\ref{E4.80}) follows from the comparison principle.

 \indent\textit{Step 2.2}
 Repeating the previous argument, we get (\ref{E4.81}) provided $\beta(e_n)\in(0,\pi/2).$

\indent\textit{Step 3. }
We prove that there exists a constant $\eta>0$, depending only on $n,\lambda,\Lambda,\theta_0$, such that either
 \begin{equation}\label{E4.91}
\partial_{\tau(-\theta_0-\eta)}u_R\geq 0\quad\mbox{in}\quad B_1,
\end{equation}
or
 \begin{equation}\label{E4.92}
\partial_{\tau(\theta_0+\eta)}u_R\geq 0\quad\mbox{in}\quad B_1,
\end{equation}
 for any $R>R_0$, where $R_0$ is given by (\ref{E4.11}).

\indent\textit{Step 3.1}
We verify (\ref{E4.91}) provided $\beta(e_n)\in(-\pi/2,0]$.

Fix $y\in B_1\cap\Omega(u_R),$ and let
\begin{equation*}
\tau\big(\beta(y)\big):=D_{T}u_R(y)/|D_{T}u_R(y)|.
\end{equation*}
Then $\beta(y)\in(-\pi/2,\pi/2)$ is implied by $\partial_nu_R(y)>0$.

First, we consider the case $\beta(y)\in(-\pi/2,0].$ Choosing $\eta\in(0,\pi/4-\theta_0/2)$,
i.e., $\angle\vartheta\Big(\tau\big(\beta(y)\big),\tau(-\theta_0-\eta)\Big)<\pi/2,$ we have
\begin{equation*}
 \partial_{\tau(-\theta_0-\eta)}u_R(y)>0.
\end{equation*}

Second, we consider the case $\beta(y)\in(0,\pi/2)$. By $\partial_{\tau(-\theta_0)}u_R(y)>0,$  we have
 \begin{equation*}
 \angle\vartheta\big(\tau\big(\beta(y)\big),\tau(-\theta_0)\big)=\beta(y)+\theta_0<\pi/2,
 \end{equation*}
i.e.,
  \begin{equation*}
  \beta(y)\in(0,\pi/2-\theta_0).
  \end{equation*}
 Using (\ref{E4.80}), we have
  \begin{equation*}
  \begin{split}
 |D_{T}u_R(y)|\cos\big(\beta(y)\big)&=\partial_{n}u_R(y)\leq N\partial_{\tau(-\theta_0)}u_R(y)\\
 &=N|D_{T}u_R(y)|\cos\big(\beta(y)+\theta_0\big).
 \end{split}
 \end{equation*}
In summary, we deduce that
 \begin{equation}\label{E4.90}
 N\cos\big(\beta(y)+\theta_0\big)\geq\cos\big(\beta(y)\big)\geq\cos\big(\pi/2-\theta_0\big)>0.
 \end{equation}

By (\ref{E4.90}), it not hard to realize that there exists $\eta_0=\eta_0(n,\lambda,\Lambda,\theta_0)>0$, such that
 \begin{equation*}
 \angle\vartheta\Big(\tau\big(\beta(y)\big),\tau(-\theta_0-\eta_0)\Big)=\beta(y)+\theta_0+\eta_0<\pi/2.
 \end{equation*}
 Taking $\eta=\min(\eta_0,\pi/4-\theta_0/2)$, then (\ref{E4.91}) follows.

\indent\textit{Step 3.2}
Repeating the previous argument, we get (\ref{E4.92}) provided $\beta(e_n)\in(\pi/2,\pi).$

 \indent\textit{Step 4. }
 For $R\gg R_0$, using (\ref{E4.91})-(\ref{E4.92}), we find that either
\begin{equation*}
0=u_R(0)\geq u_R\big(-\tau(-\theta_0-\eta)\big)=0,
\end{equation*}
  or
\begin{equation*}
0=u_R(0)\geq u_R\big(-\tau(\theta_0+\eta)\big)=0.
\end{equation*}
Then we observe that either $-\tau(-\theta_0-\eta)\in\Omega^{c}(u_R)$ or $-\tau(\theta_0+\eta)\in\Omega^{c}(u_R)$.

Taking $R\rightarrow+\infty,$ we deduce that
\begin{equation*}
\sup_{\tau\in V(u)}\angle\vartheta(-e_n,\tau)>\theta_0.
\end{equation*}
  However, this contradicts (\ref{E3.32}), as we wanted.
\end{proof}

\section{ Proof of Theorem \ref{thm:1} }

Let's begin with the following definition.

\begin{defn}
 \label{def:lower dimensional operators}
 Let $F$ be a convex operator in $\mathcal{S}_{n\times n}$ with $n\geq 2$. We define a convex operator $F_{n-1}$ by
\begin{equation*}
F_{n-1}(M):=F\begin{pmatrix} M & 0 \\ 0 & 0 \end{pmatrix}
,\qquad\forall M\in\mathcal{S}_{(n-1)\times (n-1)}.
\end{equation*}
\end{defn}

\begin{lemma}
 \label{le:4.4}
Assume all hypotheses of Corollary \ref{cor: 4.1} hold, and suppose further that $V(u)=-e_n$. Then
\begin{equation}\label{E3.7}
u_{n-1}(x'):=\lim_{x_n\rightarrow -\infty}u(x',x_n),
\end{equation}
is well defined. Furthermore, $u_{n-1}$ satisfies
\begin{equation}\label{EE4.21}
\left\{\begin{split}
F_{n-1}(D^2 u_{n-1})\leq 1\quad &\mbox{in}\quad\mathbb{R}^{n-1},\\
u_{n-1}\geq 0\quad&\mbox{in}\quad\mathbb{R}^{n-1},\\
F_{n-1}(D^2 u_{n-1})= 1\quad &\mbox{in}\quad \{u_{n-1}>0\},
\end{split}\right.
\end{equation}
in the viscosity sense.
\end{lemma}

\begin{proof} Applying Lemma \ref{le:2.4}, we see that $\partial_n u\geq 0$ follows from $-e_n\in V(u)$. Hence the nonnegative function $u_{n-1}(x')$ is well defined.
Consider the family of $(n-1)$-dimensional functions $\{u_{n-1}^k\}$,
\begin{equation*}
u_{n-1}^k(x'):=u(x',-k),\qquad k\in\mathbb{Z}^+.
\end{equation*}
Using the interior $C^{1,1}$ estimate, we have
\begin{equation*}
||D^2 u_{n-1}^k||_{L^\infty(\mathbb{R}^{n-1})}\leq||D^2 u||_{L^\infty(\mathbb{R}^n)}\leq C(n,\lambda, \Lambda).
\end{equation*}
By (\ref{E3.7}), we have
\begin{equation}\label{E4.23}
u_{n-1}^k\rightarrow u_{n-1}\geq 0\qquad\mbox{in}\quad C^{1,\alpha}_{loc}(\mathbb{R}^{n-1}),\quad\forall\alpha\in (0,1).
\end{equation}

Let
\begin{equation}\label{E4.60}
v^k(x',x_n):=u(x',x_n-k),\qquad k\in Z_+.
\end{equation}
Given $\rho>0$, using $||D^2 v^k||_{L^\infty(B_{\rho})}\leq C(n,\lambda,\Lambda)$, we have
\begin{equation}\label{E4.62}
 v^k\rightarrow v,\qquad \mbox{in}\quad C^{1,\alpha/2}(B_{\rho}).
\end{equation}
Putting (\ref{E4.23}), (\ref{E4.60}) and (\ref{E4.62}) together, we obtain
\begin{equation*}
v(x',x_n)=u_{n-1}(x'),\qquad \mbox{in}\quad B_{\rho}(0).
\end{equation*}

Using \cite[Proposition 2.9]{Ca02}, we conclude that
\begin{equation*}
F_{n-1}(D^2_{x'}u_{n-1})=F(D^2_{x}v)=\lim_{k\rightarrow+\infty}F(D^2_{x}v^k)\leq 1, \quad\mbox{for}\quad |x'|<\rho,
\end{equation*}
in the viscosity sense. Therefore, the first equation in (\ref{EE4.21}) follows.

Next we verify the third equation in (\ref{EE4.21}). Fixing $P'\in\{u_{n-1}>0\}$, we set $d:=\mbox{dist}(P',\Omega^c(u_{n-1}))>0.$
By monotonicity of $u$, we have
\begin{equation*}
v^k(x',x_n)\geq u_{n-1}(x')>0,\quad \mbox{in}\quad B_{d}(P',0)\subseteq\mathbb{R}^n.
\end{equation*}
It follows that
\begin{equation*}
F_{n-1}(D^2_{x'}u_{n-1})=\lim_{k\rightarrow+\infty}F(D^2_{x}v^k)=1, \quad\mbox{in}\quad B'_{d/2}(P'),
\end{equation*}
in the viscosity sense.
\end{proof}

 Assume that $u$ is a global solution with $V(u)=-e_n$. The following corollary states that if $\mathrm{MD}\big(\Omega^c(u)\big)<+\infty$, then $\mathrm{MD}\big(\Omega^c(u_{n-1})\big)<+\infty$.

\begin{cor}
 \label{cor:4.3}
Suppose that all hypotheses of Lemma \ref{le:4.4} hold. If
\begin{equation}\label{E3.5}
\Omega^c(u)\subseteq\bigl\{x\in\mathbb{R}^{n}:|x_1|\leq 1\bigr\},
\end{equation}
 then
\begin{equation}\label{E3.6}
\Omega^c(u_{n-1})\subseteq\bigl\{x'\in\mathbb{R}^{n-1}:|x_1|\leq 1\bigr\},
\end{equation}
where $u_{n-1}$ is defined by (\ref{E3.7}). Moreover, $\operatorname{Int}\Omega^{c}(u_{n-1})\neq\emptyset.$
\end{cor}

\begin{proof}
By $\mathrm{Int}\,\Omega^{c}(u)\neq\emptyset$ (see Corollary \ref{cor: 4.1}), we may pick $Z(z',z_n)\in$ $\mathrm{Int}\,\Omega^{c}(u)$ such that $B_{\rho}(Z)\subset\Omega^{c}(u)$,
where $\rho>0$ is a small constant. By the monotonicity of $u$, we find that
\begin{equation*}
u=0\quad\mbox{in}\quad B'_{\rho}(z')\times(-\infty,z_n).
\end{equation*}
Using $B'_{\rho}(z')\subseteq\Omega^c(u_{n-1})$, we observe that $\mathrm{Int}\,\Omega^{c}(u_{n-1})\neq\emptyset.$

Since $\Omega^c(u_{n-1})$ is a locally Lipschitz domain, which follows from the convexity of $\Omega^c(u_{n-1})$, (\ref{E3.6}) is a direct consequence of
\begin{equation}\label{E4.13}
\mathrm{Int}\,\Omega^c(u_{n-1})\subseteq \{|x_1|\leq 1\}.
\end{equation}

If (\ref{E4.13}) fails, then we can choose $P'(P'_1,\cdots,P'_{n-1})\in$ $\mathrm{Int}\,\Omega^c(u_{n-1})$ with $P'_1>1$. Taking $\rho\in (0,1)$ small if necessary, we have
\begin{equation}\label{E4.16}
B'_{2\rho}(P')\subseteq\Omega^c(u_{n-1})\cap\{x_1>1\}.
\end{equation}

By (\ref{E3.5}), we find that $P_{k}(P',-k)\in\Omega(u)$. Then we can choose a sequence $Q_k(Q_k',Q_{k,n})\in\Omega(u)$ such that
\begin{equation*}
u(Q_k)=\sup_{B_{\rho}(P_{k})}u\geq C(n,\lambda,\Lambda)\rho^2,\quad k=1,2,\cdots,
\end{equation*}
where we use the nondegeneracy condition. Up to a subsequence, we may assume that $Q'_k\rightarrow Q'\in\overline{B'_{\rho}}(P')$. Since $Q_{k,n}\rightarrow-\infty,$ we have
\begin{equation*}
u_{n-1}(Q')=\lim_{k\rightarrow +\infty}u(Q',Q_{k,n})=\lim_{k\rightarrow +\infty}u(Q_k',Q_{k,n})\geq C\rho^2.
\end{equation*}
However, this contradicts (\ref{E4.16}), as we wanted. Thus (\ref{E4.13}) follows.
\end{proof}

\begin{lemma}
 \label{lemma 4.9}
Suppose that all hypotheses of Lemma \ref{le:4.4} hold. Then
\begin{equation}\label{E4.98}
P+\mathbb{R}_+e_n\subseteq\Omega(u),\qquad\forall P\in\Gamma(u).
\end{equation}
\end{lemma}

\begin{proof}
On the contrary, we may suppose that $P\in\Gamma(u)$, and $P+e_n\in\Omega^c(u)$. Using $\partial_n u\geq 0$, we have $P+se_n\in\Omega^c(u)$ for all $s\in[0,1]$. We claim that
 \begin{equation}\label{E5.11}
P+se_n\in\Gamma(u),\quad\mbox{for all}\quad s\in[0,1].
\end{equation}

Indeed, if $P+s_0e_n\in$ $\mathrm{Int}\,\Omega^c(u)$ for some $s_0\in[0,1]$, then we take \(\rho>0\) small so that $B_{\rho}(P+s_0e_n)\subseteq\Omega^c(u)$. Using
$\partial_n u\geq 0$ again, we have $u\equiv 0$ in $B_{\rho}(P)$. This contradicts $P\in\Gamma(u)$.

 Using (\ref{E5.11}), one can check that
 \begin{equation*}
 \partial_{n}(\partial_{\nu}u)(P+e_n/2)=0,
 \end{equation*}
  where $\nu\in\partial B'_1$ denotes the unit outward normal to $\Omega^c(u)$ at $P+e_n/2$. We take $r>0$ small so that $\Omega(u)\cap B_{r}(P+e_n/2)$ is a $C^{1,\alpha}$ domain
  (see Proposition \ref{pro:2.11}). By \cite[Theorem 1.2]{LZ}, we deduce that $\partial_n u\in C^{1,\alpha}\Big(\Omega(u)\cap B_{r}(P+e_n/2)\Big)$. In summary, we have
\begin{equation*}
\partial_{\nu}(\partial_{n}u)(P+e_n/2)=\partial_{n}(\partial_{\nu}u)(P+e_n/2)=0.
\end{equation*}

However, since $\partial_n u$ is a positive solution to (\ref{E6.10}), the Hopf lemma implies that
\begin{equation*}
\partial_{\nu}(\partial_n u)(P+e_n/2)>0,
\end{equation*}
 which yields the desired contradiction.
\end{proof}

\begin{thm}
 \label{le:4.1}
Suppose that all hypotheses of Corollary \ref{cor: 4.1} hold. Then
\begin{equation*}
\mathrm{MD}(\Omega^{c}(u))=+\infty.
\end{equation*}
\end{thm}

\begin{proof}
We divide the proof into four steps.

\indent\textit{Step 1. }
Suppose on the contrary that $\mathrm{MD}(\Omega^{c}(u))<+\infty.$ We aim to show that $u$ is a lower dimensional global solution, which yields a contradiction.

Applying Theorem \ref{thm:thm4.1}, we can assume that $V(u)=\{\tau\}$. By a rescaling argument, we can assume that (\ref{E3.5}) holds. By Remark \ref{re:4.1}, we may also assume that (\ref{E6.3}) holds. Hence $\partial_n u\geq 0$. Applying Lemma \ref{le:4.4} and Corollary \ref{cor:4.3}, we see that $u_{n-1}$ is a global solution satisfying (\ref{E3.6}). Then $u_{n-1}\in\mathcal{G}_c$ provided $V(u_{n-1})=\emptyset$. If $V(u_{n-1})\neq\emptyset$, then without loss of generality we may assume that $-e_{n-1}\in V(u_{n-1}).$ Let
\begin{equation*}
u_{n-2}:=\lim_{x_{n-1}\rightarrow -\infty}u_{n-1}(x_1,\cdots,x_{n-1}).
\end{equation*}
Then $u_{n-2}$ is a global solution.

Repeating this procedure, we deduce that $u_{n-k}\in\mathcal{G}_c$ for some $1\leq k\leq n-1$. Furthermore, we have
 \begin{equation*}
 \Omega^c(u_{n-k})\subseteq\{|x_1|\leq 1\},\qquad\mathrm{Int}\,\Omega^{c}(u_{n-k})\neq\emptyset.
 \end{equation*}

\indent\textit{Step 2.}
Assume that $u_{n-1}\in\mathcal{G}_c$. we will prove that $\partial_n u\equiv 0$ when $n\geq 3.$

By Theorem \ref{thm:5}, we may select a constant $M>0$ such that
\begin{equation}\label{E5.5}
u_{n-1}\geq\left\{\begin{split}
&\frac{1}{2}x'^T A x'+Bx'-M,\quad n\geq 4,\\
&\frac{1}{2}x'^T Ax'+Bx'-M\log|x'|,\quad n=3,
\end{split}\right.\qquad\mbox{as}\quad |x'|\rightarrow+\infty.
\end{equation}
Consider
\begin{equation*}
v_R(x',x_n):=\frac{1}{2}x'^TAx'+Bx'+R^{-1/4}x_n+1,
\end{equation*}
where $R>0$ is a constant to be chosen later.

We claim that there exists $R_0>1$ such that, for all $R>R_0$,
\begin{equation}\label{E5.7}
v_R(x)\leq u(x)\qquad \mbox{on}\quad B_{R}'(-Re_n).
\end{equation}
This claim can be seen as follows. Taking $R_0$ sufficient large, we find that
\begin{equation*}
\begin{split}
u(x',-R)-v_R(x',-R)&\geq u_{n-1}(x')-v_R(x',-R)\\
&\geq -M\log R+R^{3/4}-1>0,
\end{split}
\end{equation*}
 for $|x'|\in(R^{1/4},R)$, where we use the monotonicity of $u$. By
 \begin{equation*}
\frac{1}{2}x'^TAx'+Bx'+1\leq(\mbox{trace}A)R^{1/2},\qquad \mbox{for}\quad |x'|\leq R^{1/4},
\end{equation*}
we have
\begin{equation*}
u(x',-R)-v_R(x',-R)\geq -v_R(x',-R)\geq R^{3/4}-\mbox{trace}(A)R^{1/2}>0.
\end{equation*}
This yields (\ref{E5.7}).

Fix $R>R_0$. We show that there exists $P(P', P_n)\in\partial\mathcal{C}_{R,R}(0)\backslash B_{R}'(-Re_n)$ such that
\begin{equation}\label{E5.16}
u(P)\leq v_R(P).
\end{equation}

 Indeed, if (\ref{E5.16}) fails, by (\ref{E5.7}), then we have
\begin{equation*}
\left\{\begin{split}
F(D^{2}v_R)=1  &\geq F(D^{2}u) \quad\mbox{in}\quad \mathcal{C}_{R,R},\\
v_R&\leq u\quad \mbox{on}\quad \partial \mathcal{C}_{R,R}.
\end{split}\right.
\end{equation*}
Then $v_R\leq u$ follows from Proposition \ref{pro:2.16}.

On the other hand, we observe that $u(0)=0<1=v_R(0),$ a contradiction. Hence (\ref{E5.16}) holds.

Let $P^*:=(P', P_n-R/2)$. Combining (\ref{E5.16}) and $\partial_n u\geq 0$, we deduce that there exists $P_{\xi}(P',\xi)\in\overline{P^*P}$ such that
\begin{equation}\label{E5.18}
\partial_n u(P_{\xi})=\frac{u(P)-u(P^*)}{R/2}\leq\frac{v_R(P)-u_{n-1}(P')}{R/2}\leq 4R^{-1/4},
\end{equation}
where $\overline{P^*P}$ denotes the line segment connecting $P^*$ and $P$. Taking $R\gg 1$, and recalling $V(u)=-e_n$, it is readily seen that
\begin{equation*}
\dist(z,\Omega^{c}(u))\geq\frac{R}{2},\quad\mbox{for all}\quad z\in\overline{P^*P}.
\end{equation*}
By (\ref{E5.18}), we have
\begin{equation*}
\sup_{B'_{1}(Re_n)}\partial_{n}u\leq C(n,\lambda,\Lambda)\partial_n u(P_{\xi})<C(n,\lambda,\Lambda)R^{-1/4},
\end{equation*}
where we use the Harnack inequality.

Since $u$ is convex, we deduce that
\begin{equation*}
\sup_{B_1}\partial_{n}u\leq\sup_{B'_{1}(Re_n)}\partial_{n}u<C(n,\lambda,\Lambda)R^{-1/4}.
\end{equation*}
Passing to the limit $R\rightarrow+\infty$, we then get \(\partial_nu(e_n)=0\). As $\partial_nu$ is a nonnegative solution of (\ref{E6.10}),
it follows that $\partial_n u\equiv 0$. In other words, $u$ is a lower dimensional global solution, which gives a contradiction.

\indent\textit{Step 3.}
Assuming that $u_{n-1}\in\mathcal{G}_c$, we will prove that $\partial_n u\equiv 0$ when $n=2.$

\indent\textit{Step 3.1}
Without loss of generality, we can assume that
\begin{equation}\label{E4.95}
u_1(x)=\frac{\gamma}{2}[(x+1)_{-}^{2}+(x-1)_{+}^{2}].
\end{equation}

Now define a function $f:\mathbb{R}_+\rightarrow[0,1)$ as follows:
\begin{equation*}
f(-x_2):=x_1,\qquad \forall\quad (x_1,x_2)\in\Gamma(u)\cap\{x_1\geq 0\}.
\end{equation*}
By (\ref{E4.95}), we find that $f(+\infty)=1$. The inclusion $f\in C^{1}(\mathbb{R}_+)$ follows from the local $C^{1}$ regularity of $\Gamma(u)$. The inequality $f'\geq 0$
is implied by $\partial_2 u\geq 0$, and (\ref{E4.98}) yields $f'>0$. Using the convexity of $\Omega^{c}(u)$, we conclude that $f$ is concave. In summary, we obtain
\begin{equation*}
f(0)=0,\quad f(+\infty)=1,\quad f'(t)>0,\quad f\quad\mbox{is concave}.
\end{equation*}

For sufficiently small $\epsilon>0$, there exists $R(\epsilon)\gg 1$ such that
\begin{equation*}
\epsilon>f(t)-f(t/2)=f'(\xi)\frac{t}{2},\qquad\forall t>R(\epsilon).
\end{equation*}
Since $f$ is concave, we have
\begin{equation}\label{E5.39}
f'(t)\leq f'(\xi)\leq 2\epsilon/t, \qquad \forall t>R(\epsilon).
\end{equation}

Given $\epsilon>0,$ let $R(\epsilon)$ be given by (\ref{E5.39}). Let $P(P_{1},P_{2})\in\Gamma(u)\cap\{x_1>0\}$. Assuming that $|P_{2}|>R(\epsilon)$, we claim that
\begin{equation}\label{E5.41}
0\leq\partial_{1}\partial_{2}u(P)<C(\lambda,\Lambda)\epsilon/|P_{2}|.
\end{equation}

This claim is justified as follows. Let $\delta>0$ be a small constant to be chosen later. Applying Lemma \ref{lemma 4.9}, we have $P^{\delta}(P_1,P_2+\delta)\in\Omega(u)$
since $P\in\Gamma(u)$. Recall that $L_P$ denotes the supporting line to $\Omega^c(u)$ at $P$.

We can pick $P^*(P^*_1,P_2+\delta)\in\Gamma(u)\cap\{x_1>0\}$ and $Q(Q_1,P_2+\delta)\in L_P$. By (\ref{E5.39}), we deduce that
\begin{equation*}
\begin{split}
|P_1-P^*_1|&\leq\dist(P^{\delta},Q)+\dist(Q,P^*)\\
&= f'(|P_2|)\delta+o(\delta)\leq 3\epsilon\delta/|P_2|,
\end{split}
\end{equation*}
provided that $\delta>0$ is sufficiently small.

Using the convexity of $u$ and the interior $C^{1,1}$ estimate, we have
\begin{equation*}
\begin{split}
0&\leq\partial_1 u(P^{\delta})=\partial_1 u(P^{\delta})-\partial_1 u(P^*)=\partial_{11}u(\xi)(P_1-P^*_1)\\
&<3||D^{2}u||_{L^{\infty}(\mathbb{R}^{2})}\epsilon\delta/|P_2|<C(\lambda,\Lambda)\epsilon\delta/|P_2|.
\end{split}
\end{equation*}
Combining these estimates, we have
\begin{equation*}
\partial_{2}\partial_{1}u(P)=\lim_{\delta\rightarrow 0^+}\frac{\partial_{1}u(P^{\delta})-\partial_{1}u(P)}{\delta}\leq\frac{C\epsilon}{|P_2|}.
\end{equation*}

In Lemma \ref{lemma 4.9}, we established that $\partial_2 u\in C^{1,\alpha}\big(\Omega(u)\cap B_{r}(P)\big)$ for some \(r>0\). Hence (\ref{E5.41}) from the property that \(\partial_{1}\partial_{2}u(P)=\partial_{2}\partial_{1}u(P)\).

\indent\textit{Step 3.2}
Let $P(P_{1},P_{2})\in\Gamma(u)\cap\{x_1>0\}$. We claim that
\begin{equation}\label{E5.43}
\partial_1\partial_{2}u(P)\geq C(\lambda,\Lambda)\partial_{2}u(P+\rho\nu(P))/\rho,\qquad \forall \rho>1,
\end{equation}
where $\nu(P)$ denotes the unit outward normal vector to $\Omega^{c}(u)$ at $P$.

Indeed, for $\rho>0,$ we now define
\begin{equation*}
\begin{split}
\widetilde{B'_\rho}(P):=& B_\rho(P)\cap L_P,\\
 \mathcal{Q}_{\rho}(P):=&\{y+t\nu(P):y\in\widetilde{B'_\rho}(P),t\in (0,\rho)\},\\
 \eta:=&\partial_{2}u(P+\rho\nu(P)).
\end{split}
\end{equation*}
Recall that $L_P$ denotes the supporting line to $\Omega^c(u)$ at $P$. The convexity of $\Omega^c(u)$ implies that $\mathcal{Q}_{\rho}(P)\cap\Omega^c(u)=\emptyset.$

Let $v$ satisfy
\begin{equation*}
\left\{\begin{split}
a^{ij}(x)\partial_{ij}v(x)&=0 \quad\mbox{in}\quad \mathcal{Q}_{\rho}(P),\\
v=\eta \quad\mbox{on}&\quad \widetilde{B'_\rho}(P)+\rho\nu(P),\\
v=0 \quad\mbox{on}&\quad\partial\mathcal{Q}_{\rho}(P) \backslash(\widetilde{B'_\rho}(P)+\rho\nu(P)),
\end{split}\right.
\end{equation*}
where $a^{ij}(x)=F_{ij}(D^{2}u(x))$. Since $\partial_{2}u\geq 0$ satisfies (\ref{E6.10}), we have
\begin{equation*}\label{E5.53}
\partial_2 u\geq C(\lambda,\Lambda)\eta,\qquad\mbox{on}\quad \widetilde{B'_\rho}(P)+\rho\nu(P),
\end{equation*}
by the Harnack inequality. Using the comparison principle, we observe that
\begin{equation}\label{E5.54}
0<C(\lambda,\Lambda)v\leq\partial_{2}u\quad\mbox{in}\quad\mathcal{Q}_{\rho}(P).
\end{equation}
Combining (\ref{E5.54}) and Lemma \ref{le:2.2}, we have
\begin{equation}\label{E5.56}
\partial_{\nu}(\partial_{2}u)(P)\geq C(\lambda,\Lambda)\partial_{\nu}v(P)\geq C(\lambda,\Lambda)\eta/\rho.
\end{equation}

Since $\lim\limits_{t\rightarrow\infty}f'(t)=0,$ taking $\epsilon>0$ sufficiently small, one readily sees that
\begin{equation}\label{E5.65}
(\nu(P),e_1)>1/2.
\end{equation}
Thus (\ref{E5.43}) follows from (\ref{E5.56}) and (\ref{E5.65}).

\indent\textit{Step 3.3}
 Let $\rho=|P_2|$. Together with (\ref{E5.41}) and (\ref{E5.43}), we obtain
\begin{equation*}
\partial_{2}u(P+|P_2|\nu(P))\leq C(\lambda,\Lambda)\epsilon.
\end{equation*}
By the convexity of $u$, we have
\begin{equation*}
\sup_{B_1}\partial_{2}u\leq\sup_{B'_{1}(|P_2|e_2)}\partial_{2}u\leq C(\lambda,\Lambda)\partial_{2}u(P+|P_2|\nu(P))\leq C(\lambda,\Lambda)\epsilon,
\end{equation*}
where we invoke the Harnack inequality. Letting $\epsilon\rightarrow 0^+$, we then get $\partial_2 u(e_2)=0$. Hence $\partial_2 u\equiv 0$.
In other words, $u$ is a one dimensional global solution, which yields a contradiction.

\indent\textit{Step 4.}
We finish the proof by mathematical induction. Assume inductively that \(\partial_n u\equiv 0\) whenever \(u_{n-k}\in\mathcal{G}_c\) for some \(k\geq 1\), and consider
the case \(u_{n-k-1}\in\mathcal{G}_c\). Under this inductive hypothesis, we have $\partial_{n-1}u_{n-1}\equiv 0$, that is, $u_{n-1}$ is a lower dimensional global solution.
Moreover, we obtain the asymptotic function associated with \(u_{n-1}\). Repeating the preceding arguments, one readily finds that \(\partial_n u\equiv 0\), which yields a contradiction.
\end{proof}

\begin{lemma}
 \label{lemma 4.11}
Suppose that all hypotheses of Lemma \ref{le:4.4} hold. Then there exists a unique positive definite matrix $A_{n-1}\in\mathcal{S}_{(n-1)\times (n-1)}$ such that
(\ref{E1.13}) and (\ref{E1.25}) hold.
\end{lemma}

\begin{proof}
We divide the proof into two steps.

\indent\textit{Step 1. }
Let $r_k\rightarrow+\infty$. By $||D^{2}u_{r_k}||_{L^{\infty}(\mathbb{R}^{n})}\leq C(n,\lambda,\Lambda)$, up to a subsequence, we find that
\begin{equation*}
u_{r_k}\rightarrow u_\infty,\quad\mbox{in}\quad C^{1,\alpha}_{loc}(\mathbb{R}^{n}),
\end{equation*}
as $k\rightarrow+\infty$. Since $u_\infty$ is a global solution to the obstacle problem (\ref{E1.1}), using $V(u)=-e_n$, we have
\begin{equation*}
F(D^2 u_\infty)=1,\qquad\mbox{in} \quad\mathbb{R}^{n}\backslash (-\mathbb{R}_{+}e_n).
\end{equation*}
It is readily seen that $u_\infty(x)=\frac{1}{2}x^TAx\geq 0.$ From $\partial_n u_\infty(-e_n)=A(-e_n)=0$, we have
\begin{equation*}
A=\begin{pmatrix} A_{n-1} & 0 \\ 0 & 0 \end{pmatrix}.
\end{equation*}

The uniqueness of the matrix $A_{n-1}$ follows from the argument below. For any $\rho>1$, we have
\begin{equation*}
F(D^2 u)=1,\qquad\mbox{in}\quad B_{\rho}(\rho e_n).
\end{equation*}
Applying the Evans-Krylov estimate, we deduce that
 \begin{equation*}
[D^2 u]_{C^{\alpha}(B_{\rho/2}(\rho e_n))}\leq C\rho^{-\alpha}||D^2 u||_{L^{\infty}(B_{\rho}(\rho e_n))}\rightarrow 0,\quad\mbox{as}\quad\rho\rightarrow+\infty.
 \end{equation*}
 Then $A=\lim\limits_{\rho\rightarrow+\infty}D^2 u(\rho e_n)$ is well defined.

\indent\textit{Step 2. }
In this step, we will prove that $A_{n-1}>0$.

\indent\textit{Step 2.1 }
Let
 \begin{equation*}
 v_{k}(x):=\rho_k^{-2}u(\rho_kx',\rho_kx_n-k),\quad k=1,2,\cdots,
 \end{equation*}
 where $\rho_k:=\mbox{diam}(\Omega^{c}_{k}(u))$ (see (\ref{E1.19})). By Theorem \ref{le:4.1}, we have $\mathrm{MD}(\Omega^{c}(u))=+\infty$.
Then we observe that  $\rho_k\rightarrow+\infty$. Using $V(u)=-e_n$, we deduce that
 \begin{equation}\label{E4.103}
 \lim_{k\rightarrow+\infty}\rho_k/k=0.
 \end{equation}
Applying the interior $C^{1,1}$ estimate again, we obtain
 \begin{equation}\label{E5.60}
v_{k}\rightarrow v_{\infty}\quad\mbox{in}\quad C^{1,\alpha}_{loc}(\mathbb{R}^{n}),\quad\mbox{as}\quad k\rightarrow+\infty.
 \end{equation}
 Applying \cite[Proposition 2.9]{Ca02} again, it follows that $v_{\infty}$ is a global solution to the fully nonlinear obstacle problem (\ref{E1.1}).

We claim that
 \begin{equation}\label{E4.100}
\partial_n v_{\infty}=0.
 \end{equation}
 This claim is justified as follows. Fixing $r>0$, by $\partial_{nn}u\geq 0$, we deduce that
 \begin{equation*}
\begin{split}
\partial_{n}u(\rho_k x',-k/2)\frac{k}{2}&\leq u(\rho_kx',0)-u(\rho_kx',-k/2)\\
&\leq u(\rho_kx',0)\leq C(n,\lambda,\Lambda)(r\rho_k)^2,\qquad\mbox{in}\quad B_{r},
\end{split}
\end{equation*}
 where we are use the interior $C^{1,1}$ estimate. Since
 \begin{equation*}
\rho_kx_n-k\leq -k/2,\quad\mbox{for}\quad k\gg 1,
 \end{equation*}
 choosing $k$ sufficiently large, we find that
\begin{equation*}
\begin{split}
\partial_{n}v_{k}(x)&=\rho_k^{-1}\partial_{n}u(\rho_kx',\rho_kx_n-k)\\
&\leq \rho_k^{-1}\partial_{n}u(\rho_kx',-k/2)\leq  C(n,\lambda,\Lambda)r^2(\rho_k/k),\qquad\mbox{in}\quad B_{r}.
\end{split}
 \end{equation*}
Letting $k\rightarrow+\infty$, then we obtain (\ref{E4.100}).

\indent\textit{Step 2.2 }
Let $v(x'):=v_{\infty}(x',0)$. We claim that $\Omega^{c}(v)\subseteq B'_1$.

This claim is justified as follows. By (\ref{E4.100}), we deduce that $v$ is a global solution to (\ref{EE4.21}).
Since $v$ is convex, as already noted in Corollary \ref{cor:4.3}, $\Omega^{c}(v)\subseteq B'_1$ is a direct consequence of
\begin{equation}\label{E5.62}
\mathrm{Int}\,\Omega^{c}(v)\subseteq B'_1.
\end{equation}

On the contrary, we suppose that (\ref{E5.62}) fails. Then there exists $P'\in\mathbb{R}^{n-1}$ such that
\begin{equation}\label{E5.63}
B'_{2\delta}(P')\subseteq\Omega^{c}(v)\backslash B'_1,
\end{equation}
where $\delta>0$ is a small constant. Picking $P_k(\rho_k P',-k)\in\mathbb{R}^{n}$, we have
\begin{equation*}
\dist(P_k,\mathbb{R}e_n)\geq(1+2\delta)\rho_k>\mbox{diam}(\Omega^{c}_{k}(u)),\quad k=1,2,\cdots.
\end{equation*}
It is readily seen that $P_k\in\Omega(u)$. Let
\begin{equation*}
u(Q_k):=\sup_{B_{\delta\rho_k}(P_k)}u.
\end{equation*}
Using nondegeneracy, we have
\begin{equation*}
v_{k}\big(Q'_k/\rho_k,(Q_{k,n}+k)/\rho_k\big)=\rho_k^{-2}u(Q_k)\geq C(n,\lambda,\Lambda)\delta^2.
\end{equation*}

Notice that  $\big(Q'_k/\rho_k,(Q_{k,n}+k)/\rho_k\big)\in\overline{B_{\delta}}(P',0)$. Passing to a subsequence, we assume that
$Q'_k/\rho_k\rightarrow Q'\in\overline{B'_{\delta}}(P')$. By (\ref{E4.100}), we have
\begin{equation*}
v(Q')=\lim_{k\rightarrow+\infty}v_{k}(Q'_k/\rho_k,0)\geq C\delta^2.
\end{equation*}
This contradicts (\ref{E5.63}), as we wanted.

\indent\textit{Step 2.3 }
We verify that $\mathrm{Int}\,\Omega^{c}(v)\neq\emptyset$. Let us assume by contradiction that $\mathrm{Int}\,\Omega^{c}(v)=\emptyset$. Then $v$ is a quadratic polynomial.
 By $\Omega^{c}(v)\subseteq B'_1$, we observe that $\Omega^{c}(v)$ is a single point. Since the origin $0\in\Omega^{c}(u)$, using the monotonicity of $u$,
we deduce that $v(0)=0$, i.e., $\Omega^{c}(v)=\{0\}$.

On the other hand, we can select a sequence $M_k(M'_k,-k)\in\Omega^{c}_{k}(u)$ such that
\begin{equation*}
|M'_k|=\mbox{diam}\Omega^{c}_{k}(u)=\rho_k.
\end{equation*}
Up to a subsequence, we see that $(M'_k/\rho_k)\rightarrow M'$ with $|M'|=1$. It is not hard to see that $v(M')=0,$ i.e., $M'\in\Omega^{c}(v)$, which yields the desired contradiction.

\indent\textit{Step 2.4 }
Using $\Omega^{c}(v)\subseteq B'_1$, we deduce that
\begin{equation*}
v_r(x')\rightarrow\frac{1}{2}x'^TA^*_{n-1}x',\quad\mbox{as}\quad r\rightarrow +\infty,
\end{equation*}
where the matrix $A^*_{n-1}>0$. It follows that
\begin{equation}\label{E5.66}
F_{n-1}(A^*_{n-1})=F_{n-1}(A_{n-1})=1.
\end{equation}

Since $v(x')\in\mathcal{G}_c$, for $\epsilon>0$, we may choose $R(\epsilon)>1$ large such that
\begin{equation*}
v(x')\geq\frac{1}{2}x'^{T}\left(A^*_{n-1}-\frac{\epsilon}{2}I_{n-1}\right)x',\qquad\mbox{for}\quad |x'|>R(\epsilon),
\end{equation*}
where we use Theorem \ref{thm:5} (or (\ref{E1.6})). Fix $R>R(\epsilon)$. By (\ref{E5.60}), we deduce that
\begin{equation*}
\rho_k^{-2}u(\rho_k R e-ke_n)=v_k(Re)\geq\frac{1}{2}R^2e^T(A^*_{n-1}-\epsilon I)e,\quad\forall e\in\partial B'_1,
\end{equation*}
as $k\rightarrow+\infty.$ By the monotonicity of $u$, we have
\begin{equation*}
\rho_k^{-2}u(\rho_kR e-ke_n)\leq\rho_k^{-2}u(\rho_kR e)\leq\frac{1}{2}R^2e^T(A_{n-1}+\epsilon I)e,\quad\forall e\in\partial B'_1,
\end{equation*}
as $k\rightarrow+\infty.$ Then we conclude that
\begin{equation}\label{E4.107}
A^*_{n-1}\leq A_{n-1},
\end{equation}
by taking $\epsilon>0$ sufficiently small.

Together with (\ref{E5.66}) and (\ref{E4.107}), we obtain $A_{n-1}=A^*_{n-1}>0.$
\end{proof}

\begin{proof}[Proof of Theorem \ref{thm:1}]
When $\operatorname{Int}\Omega^{c}(u)=\emptyset$, we have $u(x)=\frac{1}{2}x^TAx$. It remains to consider the case $\operatorname{Int}\Omega^{c}(u)\neq\emptyset$.

 Consider the one dimensional case $n=1$. If $\Omega^{c}(u)$ is compact, its convexity forces $\Omega^{c}(u)$ to be a closed interval. If $\Omega^{c}(u)$ is non-compact, together with $\Gamma(u)\neq\emptyset$ and the convexity of $\Omega^{c}(u)$, we find that $\Omega^{c}(u)$ is a ray in $\mathbb{R}^{1}$.
Upon translating and rotating coordinates, $u$ takes the form (\ref{E1.5}).

Next we consider the case $n\geq 2$. Under the assumption that $\operatorname{MD}(\Omega^{c}(u))<+\infty,$ invoking Theorem \ref{le:4.1}, we deduce that either
$u\in\mathcal{G}_c$, or $u$ is a lower dimensional global solution. If $u\in\mathcal{G}_c$, then $\Omega^{c}(u)$ is a compact convex set. As already mentioned in
Remark \ref{re:2.13}, the local $C^{1,\alpha}$ regularity of the free boundary follows. Using a simple compactness argument, we conclude that $\Omega^c(u)$ is a $C^{1,\alpha}$ domain.

Under the assumption that $\operatorname{MD}(\Omega^{c}(u))=+\infty,$ we have $V(u)\neq\emptyset$. Assuming that $u$ is not a lower dimensional global solution, by a suitable
translation and rotation of the coordinate system, we obtain $V(u)=\{-e_n\}$, where we appeal to Theorem \ref{thm:thm4.1}. Hence $\partial_n u\geq 0$. It is easily seen that $u\in\mathcal{G}_{\infty}$. Otherwise, suppose that $\partial_{\tau}u\geq 0$ for some $\tau\neq e_n$. Using $u(0)=0$, we get $u(-\mathbb{R}_{+}\tau)\equiv 0$, i.e.,
$-\tau\in V(u)$, a contradiction, as we wanted to see.

The inclusion (a) is trivial. We have already shown that inclusion (b) holds; see Lemma \ref{lemma 4.11}. It remains to verify (\ref{E1.18}). In fact, the first limit follows from $\operatorname{MD}(\Omega^{c}(u))=+\infty$. The second limit was established in the proof of Lemma \ref{lemma 4.11}; see (\ref{E4.103}). For $n=2$, the third limit
is trivial. For $n\geq 3$, given sufficiently small $\epsilon>0$, by (\ref{E5.60}) and (\ref{E4.100}), we find that
 \begin{equation*}
 \left\{\begin{split}
\rho_k^{-1}\mbox{diam}(\Omega^{c}_{k}(u))<\mbox{diam}(\Omega^{c}(v))+\epsilon,\\
\rho_k^{-1}\operatorname{MD}(\Omega^{c}_{k}(u))>\operatorname{MD}(\Omega^{c}(v))-\epsilon,
\end{split}\right.
\qquad\mbox{for}\quad k\gg 1.
 \end{equation*}
By Theorem \ref{thm:5}, we have
 \begin{equation*}
\mbox{diam}\big(\Omega^{c}(v)\big)/\operatorname{MD}\big(\Omega^{c}(v)\big)= C(A_{n-1},F).
\end{equation*}
Hence the third limit follows. The proof is complete.
\end{proof}

\section{ Proof of Theorem \ref{cor:3} }

Next we briefly sketch the proof of Theorem \ref{cor:3}.

(1) Suppose that all hypotheses of Theorem \ref{cor:3} hold. Then the nondegeneracy of $u$ and the interior $C^{1,1}$ estimate are known. We refer to \cite{FS01,Lee02}
for further details.

(2) If (\ref{E1.15}) holds, using the convexity of $u$, one may verify that $u$ is a half-space solution (cf.\cite{Lee02}).

(3) Assuming (\ref{E2.20}) holds, with the following Lemma \ref{le:5.2} at hand, it is standard to establish the local $C^1$ regularity of the free boundary $\Gamma(u)$. We refer to
\cite[Chapter 4.1]{PSU} for more details. Moreover, once the local $C^1$ regularity of the free boundary is established, the local $C^{1,\alpha}$ regularity follows (cf.\cite{RT}).

(4) Repeating the argument in the proof of Theorem \ref{thm:1}, we immediately obtain Theorem \ref{cor:3}.

\begin{lemma}
 \label{le:5.2}
Let $F$ be a concave operator, and let $u$ be a viscosity solution to the obstacle problem in $B_1$. Suppose that $u$ is a convex function. Suppose further that
$0\leq\epsilon_0\leq\lambda/(8n\Lambda^2)$. If
\begin{equation*}
C_0\partial_e u-u\geq -\epsilon_0\qquad\mbox{in}\qquad B_{1},
\end{equation*}
for some $C_0\geq 0$, where $e\in\partial B_1$, then
\begin{equation*}
C_0\partial_e u-u\geq 0\qquad\mbox{in}\qquad B_{1/2}.
\end{equation*}
\end{lemma}

\begin{proof}
 Now let us assume by contradiction that there exists $x_0\in B_{1/2}$ such that
\begin{equation*}\label{E7.21}
C_0\partial_e u(x_0)-u(x_0)<0,
\end{equation*}
and consider the function
\begin{equation*}\label{E7.22}
w(x)=C_0\partial_e u(x)-u(x)+\frac{\lambda|x-x_0|^2}{2n\Lambda^2}.
\end{equation*}

By the convexity of $u$, we have
\begin{equation*}\label{E7.23}
1=F(D^{2}u)\leq \mathcal{M}^+(D^{2}u)\leq \Lambda\mbox{tr}(D^{2}u)\qquad\mbox{in}\quad \Omega.
\end{equation*}
Then we deduce that
\begin{equation*}\label{E7.24}
a^{ij}\partial_{ij}u\geq\mathcal{M}^-(D^{2}u)\geq\lambda\mbox{tr}(D^{2}u)\geq\frac{\lambda}{\Lambda}\qquad\mbox{in}\quad \Omega,
\end{equation*}
where $a^{ij}=F_{ij}(D^{2}u)$. Since $\partial_e u$ solves the linear equation $a^{ij}\partial_{ij}v=0$, we obtain
\begin{equation*}\label{E7.25}
a^{ij}\partial_{ij}w\leq -a^{ij}\partial_{ij}u+\mathcal{M}^+(\frac{\lambda}{n\Lambda^2}I_{n\times n})\leq 0
\qquad\mbox{in}\quad \Omega.
\end{equation*}
Hence, by the maximum principle,
\begin{equation*}\label{E7.26}
\min_{\partial(\Omega\cap B_1)}w=\min_{\Omega\cap B_1}w\leq w(x_0)<0,
\end{equation*}
since $w\geq 0$ on $\partial\Omega\cap B_1$. Using $|x-x_0|^2\geq 1/4$ on $\partial B_1$, it follows that
\begin{equation*}
\lambda/(8n\Lambda^2)-\epsilon_0\leq\min_{\partial B_1\cap\Omega}w<0,
\end{equation*}
which yields a contradiction provided $\epsilon_0\leq \lambda/(8n\Lambda^2)$.
\end{proof}

\bibliographystyle{elsarticle-num}

\end{document}